\documentclass[11pt]{amsart}
\usepackage{amsaddr}
\usepackage{multicol}
\usepackage[english]{babel}
\usepackage{latexsym}
\usepackage{amssymb}
\usepackage{amsfonts}
\usepackage{amsthm}
\usepackage{graphicx}
\usepackage{mathrsfs}
\usepackage{dsfont}
\usepackage{amscd}
\usepackage{amsmath}
\usepackage{cancel}
\usepackage{multirow}
\usepackage{xcolor}
\usepackage{mathtools}
\usepackage{float}
\usepackage{wrapfig}
\usepackage[utf8]{inputenc}
\usepackage{chngcntr}
\usepackage{apptools}
\AtAppendix{\counterwithin{theorem}{section}}
\AtAppendix{\counterwithin{equation}{section}}
\numberwithin{equation}{section}

\newtheorem{theorem}{Theorem}

\newtheorem{lemma}[theorem]{Lemma}

\newtheorem{proposition}{Proposition}[theorem]
\DeclareMathOperator\erfi{erfi}
\DeclareMathOperator\erf{erf}
\begin{document}

\title[Self-Similar Solutions to the Mullins' Equation]{Self-Similar Grooving Solutions to the Mullins' Equation}

\author{Habiba V.~Kalantarova$^\ast$$^1$}
\email{$^1$kalantarova@campus.technion.ac.il}
\author{Amy Novick-Cohen$^\ast$$^2$}
\email{$^2$amync@technion.ac.il}
\address{$^\ast$Department of Mathematics, Technion-IIT, Haifa 32000, Israel}

\date{\today}

\begin{abstract}
In 1957, Mullins proposed  surface diffusion motion  as a model for
 thermal grooving. By adopting a small slope approximation,
 he reduced the model to the   \textit{Mullins' linear surface diffusion equation,}
 \begin{equation} \nonumber
 ({\rm{ME}})\quad\quad  y_t + B y_{xxxx}=0,
 \end{equation}
known also more simply as the \textit{Mullins' equation}. Mullins sought self-similar solutions to (ME) for planar initial conditions, prescribing  boundary conditions at   the thermal groove, as well as far field decay. He found explicit series solutions which are routinely used in analyzing thermal grooving to this day.

While (ME) and the small slope approximation are physically reasonable, Mullins' choice of boundary conditions is not always appropriate. Here  we present an in depth study of self-similar solutions to  the Mullins' equation  for general self-similar boundary conditions, explicitly identifying four linearly independent solutions defined on $\mathbb{R}\setminus\{0\}$; among these four solutions,
two exhibit unbounded growth and two exhibit asymptotic decay, far from the origin. We indicate how the full set of solutions can be used in analyzing the effective boundary conditions from experimental profiles and in evaluating the governing  physical parameters.
\end{abstract}

\maketitle

\section{Introduction}\label{sec:intr}

Motion by surface diffusion
\begin{equation} \label{msd}
V_n = - B \triangle_s \kappa,
\end{equation}
describes a geometric motion for an evolving surface. In (\ref{msd}),  $V_n$ and $\kappa$ denote, respectively, the normal velocity and the mean curvature of the evolving surface, $\triangle_s$ denotes the Laplace-Beltrami operator known also as the  surface Laplacian, and $B$ is the \textit{Mullins'  coefficient}. Motion by surface diffusion, as well as motion by mean curvature, were first proposed by Mullins  \cite{Mullins1956,Mullins1957} in  modeling the evolution of microstructure in polycrystalline materials. Polycrystalline materials contain numerous crystals or \textit{grains}, separated by grain boundaries, and there is a tendency for thermal grooves to form where interior grain boundaries intersect the exterior surface of the polycrystalline specimen. The
  evolution of the microstructure, including the phenomenon of thermal grooving, are of quite general interest,  since the microstructure and grooving in particular are highly
  influential in determining the strength,  the stability, as well as many other properties of polycrystalline materials.

  In using (\ref{msd}) to  model the development of thermal grooves, various possible effects have been neglected, such as bulk diffusion \cite{Hardy1991},    surface energy anisotropy \cite{Davi1990}, \cite{Klinger2001}, \cite{Xin2003}, as well as  evaporation and condensation \cite{Mullins1957}.
  The Mullins' coefficient is frequently prescribed as  $B= D_s \gamma_{ext}\, \Omega^2\nu/(kT)$, where
 $D_{s}$ is the surface diffusion coefficient, $\gamma_{ext}$ is the surface-free energy per unit area of the exterior surface, $\Omega$ is the atomic volume,  $\nu$ is the number of mobile atoms per unit area, $k$ is the Boltzmann constant and   $T$ is the temperature.

\bigskip
In studying the formation of thermal grooves, it is constructive to focus on the normal cross-section to some particular thermal
groove. Under the assumption that the height of the exterior surface can be described in the normal cross-section as the graph of function, $y=y(t,x)$,
relative to  an initially planar exterior surface, $y(0,x)\equiv 0,$  and
 that there is little out of plane variation in the shape of the thermal groove relative to the cross-sectional plane, then (\ref{msd}) implies that
\begin{equation}\label{mllnsoe}
y_{t}=-B[\kappa_{x}(1+y_{x}^{2})^{-1/2}]_{x}, \kappa=y_{xx}(1+y_{x}^{2})^{-3/2},\ x\in \mathbb{R}\setminus \{0\},\ t>0.
\end{equation}
In writing (\ref{mllnsoe}), it has  been implicitly assumed that the thermal groove is initially located at $x=0$ and maintains its location there, and that there are no
additional effects  influencing the shape of the exterior surface.

 To obtain a complete problem formulation for (\ref{mllnsoe}), it is reasonable to impose conditions at $x=0$ as well as far field conditions, in addition to the initial planarity condition $y(0,x)\equiv 0$. Mullins \cite{Mullins1957} effectively imposed symmetry with respect to $x=0$, implying that the grain boundary attached below the thermal groove is constrained to lie along the $y-$axis and remain
 orthogonal to the planar surface $y\equiv 0$  for $t \ge 0$.
In accordance with balance of mechanical forces (Herring's law), he required that
$y_x(t,0^+):=\lim_{x\rightarrow0^{+}}y_{x}(t,x)=m/\sqrt{4-m^2}$, where $m=\gamma_{gb}/\gamma_{ext}$ and $\gamma_{gb}$, $\gamma_{ext}$ denote, respectively, the surface energies of the grain boundary and of the exterior surface. Zero mass flux along the thermal groove was assumed. Noting that the resultant problem was non-trivial,  Mullins observed that typically $0< m=\gamma_{gb}/\gamma_{ext}< 1/3.$ This allowed him to treat $m$ as a small dimensionless parameter and to make the
 physically reasonable assumption that the slope of the exterior surface remained small at all times.

\bigskip
Based on the {small slope assumption}, Mullins \cite{Mullins1957} obtained a simpler linear problem formulation, namely, Mullins' linear surface diffusion equation
\begin{equation}\label{b2}
y_{t}+B y_{xxxx}=0,  \quad x\in \mathbb{R}\setminus \{0\},\quad t>0,
\end{equation}
often referred to more simply as the \textit{Mullins' equation},  (ME), together with the initial condition
\begin{equation}\label{icb2}
y(0,x)=0, \quad  x\in \mathbb{R},
\end{equation}
the boundary conditions at $x=0$,
\begin{equation} \label{bcb2}
\lim_{x\rightarrow0^{\pm}}y_x(t,x)=\pm m/2,\  \lim_{x\rightarrow0^{\pm}}y_{xxx}(t,x)=0, \quad t>0,
\end{equation}
as well as far field decay.\footnote{In \cite{Mullins1957} the assumption is made that the "solution $\ldots$ behaves properly at infinity."}
Mullins sought symmetric self-similar solutions of the form
 \begin{equation} \label{ssform}
y(t,x) = (Bt)^{1/4} Z(x/(Bt)^{1/4}),
\end{equation}
where $Z=Z(u)$ satisfies
\begin{equation}\label{rm34}
Z^{(4)}(u)-\frac{1}{4}uZ'(u)+\frac{1}{4}Z(u)=0,\quad u\in\mathbb{R},
\end{equation}
for the problem prescribed in (\ref{b2})--(\ref{bcb2}),  guided by the form of the Laplace transform of (\ref{b2}). He obtained a power series solution with  recursively defined coefficients, and this solution implied the now classical formula for the depth
of the thermal groove as a function of time, $d(t):=y(t,0),$  namely
\begin{equation} \label{depth}
d(t)= -\frac{m (Bt)^{1/4}}{2\sqrt{2}\Gamma(5/4)},\quad t\geq 0.
\end{equation}
 Studies of this problem in the physical literature typically rely strongly  on the linear solution derived by Mullins. In \cite{Martin2009}, Martin obtained an integral representation for Mullins' solution by using Fourier cosine transforms, which led him to conclude that Mullins' solution exhibited far field decay to planarity.

\bigskip

Often Mullins' assumptions regarding the accompanying boundary conditions are not overly accurate.
 Possible  concerns in this direction include the following:
 The underlying grain boundary may not remain vertical due to internal motion of the grain boundaries, hence the symmetry assumption may not be valid. Often there
 is some amount of mass flux along the grain boundary which reaches and interacts with the thermal groove, so the vanishing  mass flux assumption may not be realistic, \cite{Amram2014}.
 Since all specimens are necessarily of finite extent,
 far field planarity is not obvious, and it often is of interest to analyze the development of thermal grooves which are not well isolated from their
 surroundings. Accordingly with these issues in mind, we return in this paper to consider self-similar solutions to (ME) on $(t,x)\in (0, T)\times \mathbb{R}\setminus\{0\}$ with more general boundary conditions, without explicitly imposing far field decay or initial planarity.

We begin by treating the  resultant more general problem by making  use of the theory of generalized hypergeometric differential equations (GHDE), \cite{NIST}, to demonstrate that all self-similar solutions to \eqref{b2} of the form \eqref{ssform} may be expressed as
\begin{equation}\label{d41}
y(t,x)=(Bt)^{1/4} \sum_{i=0}^{3} C_i\, z_i(u), \quad\quad u=x/(Bt)^{1/4}, \quad C_i \in\mathbb{R},
\end{equation}
where
\begin{eqnarray}
&z_{1}(u)=\prescript{}{1}{F}_{3}^{}(-\frac{1}{4};\frac{1}{4},\frac{1}{2},\frac{3}{4};\frac{u^{4}}{256}),\quad  &z_{2}(u)=u, \nonumber\\[2ex]
&z_{3}(u)=u^{2}\prescript{}{1}{F}_{3}^{}(\frac{1}{4};\frac{3}{4},\frac{5}{4},\frac{3}{2};\frac{u^{4}}{256}),\quad  & z_{4}(u)=u^{3}\prescript{}{1}{F}_{3}^{}(\frac{1}{2};\frac{5}{4},\frac{3}{2},\frac{7}{4};\frac{u^{4}}{256}),\nonumber
\end{eqnarray}
and the functions $\prescript{}{1}{F}_{3}^{}(a_1;b_1,b_2,b_3;\cdot)$ with $a_1,\, b_1,\, b_2,\, b_3 \in \mathbb{R}$ denote generalized hypergeometric functions.
 The functions $\{z_i(u)\}_{i=1}^{i=4}$ defined above are linearly independent entire functions which satisfy \eqref{rm34}; moreover,  $z_1(u)$, $z_3(u)$ are even  and
 $z_2(u)$, $z_4(u)$  are odd.  It can also be readily shown that
 \begin{equation} \label{zij}
 z_i^{(j-1)}(0)=\delta_{i\,j} (j-1)!, \quad i,j=1,2,3,4.
 \end{equation}
 From (\ref{ssform}), (\ref{d41}), (\ref{zij}), it follows that
 \begin{equation} \label{bcz}
 C_{i}=\frac{1}{(Bt)^{(2-i)/4} (i-1)!}\frac{\partial^{(i-1)} y(t,0)}{\partial x^{(i-1)}} , \quad i=1,2,3,4.
 \end{equation}

Recalling  that by assumption, in our geometry a thermal groove is forming at $x=0$, which effectively reflects the development of a singularity, 
 it is reasonable to consider the behavior of solutions on either side of the thermal groove separately. This leads us to define for $t>0$ 
 \begin{equation}\label{solution_lin}
   y(t,x) =
  \begin{cases}
    (Bt)^{1/4}\sum_{i=1}^{4} C_{i}^{+} z_i(u),      & \quad C_{i}^{+}\in\mathbb{R},\ x>0,\\[1ex]
    (Bt)^{1/4}\sum_{i=1}^{4} C_{i}^{-} z_i(u),      & \quad C_{i}^{-}\in\mathbb{R},\ x<0,
  \end{cases}
\end{equation}  
where $u=\frac{x}{(Bt)^{1/4}}$ and
 \begin{equation} \label{bc_lin}
C_{i}^{\pm}=\frac{1}{(Bt)^{(2-i)/4} (i-1)!} \lim_{x\rightarrow 0^{\pm}}\frac{\partial^{(i-1)} y(t,x)}{\partial x^{(i-1)}},\quad i=1, 2, 3, 4.
 \end{equation}
Note in particular that it follows from (\ref{bc_lin}) that the coefficients in (\ref{solution_lin}) are directly proportional to the derivatives of the surface profile at
the thermal groove. This feature makes the solution representation (\ref{solution_lin}) useful for data fitting.

An alternative approach to solving the Mullins' equation is via Laplace transform methods under the  assumption of initial planarity and general boundary conditions at zero in accordance with the self-similar form \eqref{ssform}. Proceeding in this fashion yields
four linearly independent self-similar solutions of the form \eqref{ssform}, which we denote by $\{y_i\}_{i=1}^{i=4}$. Recalling \eqref{rm34} and \eqref{d41}, it follows that the set of self-similar solutions of the form \eqref{ssform} to (ME) is spanned by four linearly independent functions. Hence each of the functions $y_{i}(t,x)$, $i=1,2,3,4,$ may be expressed as a linear combination of the functions $\{(Bt)^{1/4}z_{i}(u)\}_{i=1}^{4}$, where $u=\frac{x}{(Bt)^{1/4}}$.
Accordingly by evaluating 

\begin{equation}
\frac{\partial^{(j-1)}y_{i}(t,0)}{\partial x^{(i-1)}},\quad i, j=1, 2, 3, 4\nonumber
\end{equation}

\noindent from their Laplace transforms and taking \eqref{zij} into consideration, we find for $t>0$ and $x\in\mathbb{R}\setminus\{0\}$ that

\begin{equation}
\begin{bmatrix}
y_{1}(t,x)\\[0.3em]
y_{2}(t,x)\\[0.3em]
y_{3}(t,x)\\[0.3em]
y_{4}(t,x)
\end{bmatrix}
=(Bt)^{1/4}\begin{bmatrix}
 0    & \frac{1}{\sqrt{2}}  & -1  &\frac{1}{\sqrt{2}} \\[0.5em]
 1    & -\frac{1}{\sqrt{2}} &  0  &\frac{1}{\sqrt{2}} \\[0.5em]
 0    & \frac{1}{\sqrt{2}}  &  1  &\frac{1}{\sqrt{2}} \\[0.5em]
 1    &\frac{1}{\sqrt{2}}   &  0  &-\frac{1}{\sqrt{2}}
\end{bmatrix}
\begin{bmatrix}
 \frac{1}{\Gamma\left(\frac{5}{4}\right)}z_{1}(u)\\[0.3em]
 \frac{1}{\Gamma(1)} z_{2}(u)\\[0.3em]
 \frac{1}{2\Gamma\left(\frac{3}{4}\right)}z_{3}(u)\\[0.3em]
 \frac{1}{6\Gamma\left(\frac{1}{2}\right)}z_{4}(u)
\end{bmatrix},\nonumber
\end{equation}

\noindent where $u=\frac{x}{(Bt)^{1/4}}$. Here as in \eqref{solution_lin}-\eqref{bc_lin} we may define solutions separately on either side of the thermal groove.

It is easy to show that
\begin{equation}
y_{1}(t,x)=-y_{3}(t,-x)\quad\mbox{ and }\quad y_{2}(t,x)=y_{4}(t,-x)\nonumber
\end{equation}
for $t>0$, $x\in\mathbb{R}\setminus\{0\}$. Moreover for $t>0$, $\{y_{i}(t,x)\}_{i=1}^{i=2}$ and  $\{y_{i}(t,-x)\}_{i=3}^{i=4}$ are asymptotically flat as $x\rightarrow\infty$, and for $x>0$, $\{y_{i}(t,x)\}_{i=1}^{i=2}$ and  $\{y_{i}(t,-x)\}_{i=3}^{i=4}$ satisfy the initial planarity  condition. By examining series solution expressions for  $\{y_i(t,x)\}_{i=3}^{i=4}$ and $\{y_{i}(t,-x)\}_{i=1}^{i=2}$ for $t>0$, we show that they exhibit unbounded far field growth as $x\rightarrow\infty$, and that they do not satisfy initial planarity.

Martin \cite{Martin2009} demonstrated that it is  possible to obtain an integral representation for Mullins' solution  by using Fourier cosine transforms;
we demonstrate that it is possible to obtain two linearly independent solutions by Fourier cosine transform method, and the integral representations obtained by this approach for these solutions allow us to ascertain that both solutions tend to zero as $x\rightarrow\infty$, for fixed $t>0$.

\bigskip
Earlier we mentioned that  the solution representation given in (\ref{solution_lin}) is useful for data fitting.
By undertaking a direct statistical least squares comparison with experimental data from Amram et al., \cite{Amram2014}, we show in Section \ref{sec:df} that our results can be successfully used to fit data and to distinguish between experiments in which Mullins' boundary conditions are accurate from experiments in which other boundary conditions such as the boundary conditions proposed in Amram et al.~\cite{Amram2014}, namely
\begin{equation}\nonumber
y_x(t,0)=m/2, \quad y_{xx}(t,0)=0, \quad \lim_{x \rightarrow \infty} y(t,x)=0,\quad t>0,
\end{equation}
are more accurate. It is not difficult to verify that Mullins' solution \cite{Mullins1957} may be expressed in terms of the functions $\{y_{i}(t,x)\}_{i=1}^{4}$  as
\begin{equation}\label{msol}
\frac{m}{2\sqrt{2}}(y_{1}(t,x)-y_{2}(t,x)),
\end{equation} 
 and that the solution discussed by Amram et al. \cite{Amram2014} corresponds to
\begin{equation}\label{asol}
-\frac{m}{\sqrt{2}}y_{2}(t,x),
\end{equation}
for details see Section \ref{asymp_dec}. Essentially, our results yield a method for "reading off" the effective boundary conditions from the measurements, see \cite{Kalantarova2019data}. To the authors' knowledge, this is the first study presenting an analytical solution to  the Mullins' equation that makes such data fitting possible, which is perhaps the main advantage of our approach.

\bigskip

The paper is organized as follows. In Section 2, we state and prove our main results  regarding  the existence of a four dimensional self-similar solution to (ME), which can be expressed in terms of generalized hypergeometric solutions
 as well as via Laplace transforms, and we describe the far field behavior of these solutions. In Section 3, we briefly demonstrate how these results can be used for data fitting. In Appendix A, we discuss the generalized hypergeometric solutions, indicating in detail how a specific GHDE may be identified whose solutions yield $\{z_{i}(u)\}_{i=1}^{4}$. In Appendix B, we prove in detail the initial and far field behavior of the solutions, as well as demonstrating \eqref{msol}, \eqref{asol}.

\section{Self-similar Solutions to the Mullins' Equation}\label{sec:lt}
Mullins' linear surface diffusion equation (ME)
\begin{equation}
\label{m22}y_{t}+B y_{xxxx}=0,\quad x\in\mathbb{R}\setminus\{0\}, \quad t>0,
\end{equation}
along with the initial and boundary conditions
\begin{eqnarray}
&&\label{lic}y(0,x)=0\quad x\in\mathbb{R}\setminus\{0\},\\
&&\label{slope}\lim_{x\rightarrow0^{\pm}}y_x(t,x)=\pm\frac{m}{2},\quad t>0,
\end{eqnarray}
has the following scaling symmetry, namely, given any solution $y(t,x)$ to \eqref{m22}-\eqref{slope},
\begin{equation}
y_{\lambda}(t,x)=\lambda^{-1}y(\lambda^{4}t, \lambda x)\nonumber
\end{equation}
\noindent is also a solution of \eqref{m22}, for any $\lambda>0$. This scaling property leads one to seek similarity solutions of the form, 
\begin{equation}\label{lt4}
y(t,x)=(Bt)^{1/4}Z\left(\frac{x}{(Bt)^{1/4}}\right),
\end{equation}
\noindent see \cite{Bluman2010}, \cite{Mullins1957}. The nonlinear problem \eqref{mllnsoe}, \eqref{lic} also has this scaling property, \cite{Derkach2018, Mullins1957}, but our focus here is on similarity solutions for the linear problem.

Substituting \eqref{lt4} into \eqref{m22} and making the change of variable  $u=\frac{x}{(Bt)^{1/4}}$, yields that
\begin{equation}
\label{m34}Z^{(4)}(u)-\frac{1}{4}uZ'(u)+\frac{1}{4}Z(u)=0,\quad u\in\mathbb{R}.
\end{equation}
Having obtained \eqref{m34}, Mullins \cite{Mullins1957} went on to look for a power series solution assuming zero flux at the groove root and far field decay and calculated its coefficients. Here we consider \eqref{m34} without imposing further restrictions on the solutions, $Z$, of \eqref{m34}, such as no flux or decay, and this allows us to gain a more complete understanding of \eqref{m34} and its solutions.

\begin{theorem}
The fourth order linear ordinary differential equation \eqref{m34}
\begin{equation}
Z^{(4)}(u)-\frac{1}{4}uZ'(u)+\frac{1}{4}Z(u)=0,\quad u\in\mathbb{R},\nonumber
\end{equation}
has the following fundamental set of solutions
\begin{eqnarray}
\label{d9a}&&z_{1}(u)=\prescript{}{1}{F}_{3}^{}(-\frac{1}{4};\frac{1}{4},\frac{1}{2},\frac{3}{4};\frac{u^{4}}{256}),\\
\label{d9b}&&z_{2}(u)=u,\\
\label{d9c}&&z_{3}(u)=u^{2}\prescript{}{1}{F}_{3}^{}(\frac{1}{4};\frac{3}{4},\frac{5}{4},\frac{3}{2};\frac{u^{4}}{256}),\\
\label{d9d}&&z_{4}(u)=u^{3}\prescript{}{1}{F}_{3}^{}(\frac{1}{2};\frac{5}{4},\frac{3}{2},\frac{7}{4};\frac{u^{4}}{256}),
\end{eqnarray}
where $\prescript{}{p}{F}_{q}^{}(a_{1}\ldots,a_{p};b_{1},\ldots,b_{q};u)$ for $\{a_{i}\}_{i=1}^{p}$, $\{b_{i}\}_{i=1}^{q}\in\mathbb{R}$ denotes the generalized hypergeometric function (or the generalized hypergeometric series) defined as
\begin{multline}\label{def:hf}
\prescript{}{p}{F}_{q}^{}(a_{1}\ldots,a_{p};b_{1},\ldots,b_{q};\nu)=\sum_{k=0}^{\infty}\frac{(a_{1})_{k}\ldots(a_{p})_{k}}{(b_{1})_{k}\ldots(b_{q})_{k}}\frac{\nu^{k}}{k!}\\
=1+\frac{a_{1}\ldots a_{p}}{b_{1}\ldots b_{q}}\nu+\frac{a_{1}(a_{1}+1)\ldots a_{p}(a_{p}+1)}{b_{1}(b_{1}+1)\ldots b_{q}(b_{q}+1)2!}\nu^{2}+\ldots,
\end{multline}
in which
\begin{equation}
(\lambda)_{k}=\frac{\Gamma(\lambda+k)}{\Gamma(\lambda)}=\lambda(\lambda+1)\ldots(\lambda+k-1)\nonumber
\end{equation}
is the Pochhammer symbol.
\end{theorem}

\begin{proof}
Employing the change of variable
\begin{equation}\label{d33}
u(v)=4v^{1/4}
\end{equation}
in equation \eqref{m34}, yields the equation
\begin{equation}\label{d43}
[v\frac{d}{dv}(v\frac{d}{dv}+b_{1}-1)(v\frac{d}{dv}+b_{2}-1)(v\frac{d}{dv}+b_{3}-1)-v(v\frac{d}{dv}+a_{1})]V=0,
\end{equation}
for $V(v)=Z(u)\big|_{u=u(v)}$, where
\begin{equation}\label{d21}
a_{1}=-\frac{1}{4},\quad b_{1}=\frac{1}{4},\quad b_{2}=\frac{1}{2},\quad b_{3}=\frac{3}{4}.
\end{equation}
Equation \eqref{d43} constitutes a generalized hypergeometric equation.

Generalized hypergeometric differential equations (\textbf{GHDE})

\begin{equation}\label{d36}
v\frac{d}{dv}\left(\prod_{i=1}^{q}(v\frac{d}{dv}+b_{i}-1)\right)V-v\left(\prod_{j=1}^{p} (v\frac{d}{dv}+a_{j})\right)V=0,
\end{equation}
for $V=V(v)$, where  $p,\,q \in \mathbb{Z}_+$, $p, \, q>2,$  and $\{a_{j}\}_{j=1}^{j=p}, \; \{b_{i}\}_{i=1}^{i=q}, v  \in \mathbb{C}$, were first studied by Thomae \cite{Thomae1870}. In particular, Thomae showed that there
exists a solution to equation \eqref{d36}, which he denoted as 
\begin{equation}
\prescript{}{p}{F}_{q}^{}(a_{1},\ldots,a_{p};b_{1},\ldots,b_{q};v).\nonumber
\end{equation}

Accordingly, it follows from \eqref{d43}-\eqref{d21} that
\begin{equation}\label{d34}
\prescript{}{1}{F}_{3}^{}(-\frac{1}{4};\frac{1}{4},\frac{1}{2},\frac{3}{4};\frac{u^{4}}{256})
\end{equation}
is a solution to \eqref{m34}.

From the theory of generalized hypergeometric equations, see  e.g. \cite[Chapter 16]{NIST}, since $p<q$  in (\ref{d21}) and since  none of the differences between the numbers $0, b_{1}, b_{2}, b_{3}$ is an integer, it follows that 
\begin{equation}
V_0(v)=\prescript{}{1}{F}_{3}^{}(a_{1};b_{1},b_{2},b_{3};v)\nonumber
\end{equation}
together with
\begin{equation}\label{d35}
\quad V_j(v)= v^{1-b_{j}}\prescript{}{1}{F}_{3}^{}(1+a_{1}-b_{j}; 1+b_{1}-b_{j},\ldots,\ast,\ldots,1+b_{q}-b_{j};v),
\end{equation}
where $j=1,2,3,$ and where $\ast$ indicates that the $j$th entry is replaced by $2-b_{j},$ form a fundamental set of linearly independent solutions to \eqref{d43}, \eqref{d21}, for $v\in\mathbb{C}$.

 Recalling the change of variables (\ref{d33}), it now follows that
\begin{eqnarray}
&&\label{d37}z_{1}(u)=\prescript{}{1}{F}_{3}^{}(-\frac{1}{4};\frac{1}{4},\frac{1}{2},\frac{3}{4};\frac{u^{4}}{256})=\sum_{k=0}^{\infty}\frac{\left(-\frac{1}{4}\right)_{k}}{\left(\frac{1}{4}\right)_{k}\left(\frac{1}{2}\right)_{k}\left(\frac{3}{4}\right)_{k}256^{k}k!}u^{4k},\\
&&\label{d38}z_{2}(u)=u\prescript{}{1}{F}_{3}^{}(0;\frac{1}{2},\frac{3}{4},\frac{5}{4};\frac{u^{4}}{256})= u,\\
&&\label{d39}z_{3}(u)=u^{2}\prescript{}{1}{F}_{3}^{}(\frac{1}{4};\frac{3}{4},\frac{5}{4},\frac{3}{2};\frac{u^{4}}{256})=\sum_{k=0}^{\infty}\frac{\left(\frac{1}{4}\right)_{k}}{\left(\frac{3}{4}\right)_{k}\left(\frac{5}{4}\right)_{k}\left(\frac{3}{2}\right)_{k}256^{k}k!}u^{4k+2},\\
&&\label{d40}z_{4}(u)=u^{3}\prescript{}{1}{F}_{3}^{}(\frac{1}{2};\frac{5}{4},\frac{3}{2},\frac{7}{4};\frac{u^{4}}{256})=\sum_{k=0}^{\infty}\frac{\left(\frac{1}{2}\right)_{k}}{\left(\frac{5}{4}\right)_{k}\left(\frac{3}{2}\right)_{k}\left(\frac{7}{4}\right)_{k}256^{k}k!}u^{4k+3},
\end{eqnarray}
form a fundamental set of solutions to \eqref{m34} for $u \in \mathbb{R}$.
\end{proof}
\hfill

The elements of the fundamental set of solutions $\{z_{i}(u)\}_{i=1}^{4}$ of \eqref{m34} are portrayed in Fig.\ref{fig:fss}. They  converge for all finite values of $u\in\mathbb{R}$ and define entire functions, \cite{NIST}. Moreover they exhibit the following asymptotic behavior

\begin{figure}[h]
\centering
\begin{center}
  \includegraphics[width=1.0\linewidth]{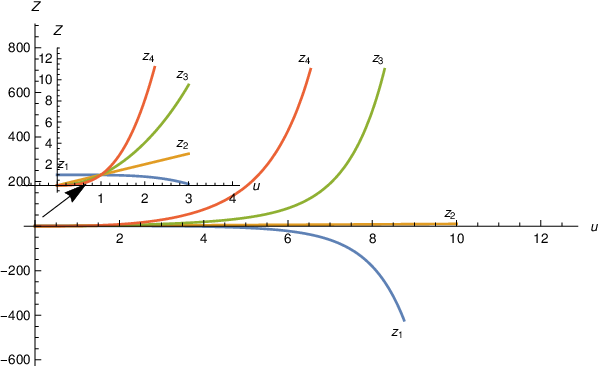}
\end{center}
\caption{\footnotesize{The elements of the fundamental set of solutions.}}
\label{fig:fss}
\end{figure}
\begin{flushleft}
\begin{eqnarray}
&&\lim_{u\rightarrow\infty}z_{1}(u)=-\infty,\quad\quad \lim_{u\rightarrow\infty}z_{i}(u)=\infty,\quad i=2,3,4,\nonumber\\
&&\lim_{u\rightarrow-\infty}z_{i}(u)=-\infty,\quad i=1,2,4,\quad\quad \lim_{u\rightarrow-\infty}z_{3}(u)=\infty.\nonumber
\end{eqnarray}
\end{flushleft}

From the definitions \eqref{d37}-\eqref{d40}, it is easy to verify that $z_{i}^{(j-1)}(0)$, $i, j= 1,2,3,4,$ satisfies \eqref{zij}. Returning to \eqref{lt4}, \eqref{m34}, we obtain
\begin{theorem}
If $y(t,x)$ is a self-similar solution to \eqref{m22} for $x>0$ (or $x<0$), $t>0$, which is of the form \eqref{lt4}, then $y(t,x)$ may be expressed as
\begin{equation}\label{d10}
y(t,x)=(Bt)^{1/4}\sum_{i=1}^{4}C_{i}z_{i}\left(\frac{x}{(Bt)^{1/4}}\right),\quad x>0,\ (\mbox{or }x<0),\ t>0,
\end{equation}
where $C_{i}$, $i=1,\ldots, 4$ are arbitrary constants, and the functions $\{z_{i}\}_{i=1}^{4}$ are prescribed in \eqref{d37}-\eqref{d40}.
\end{theorem}

The following theorem allows us to distinguish between decaying and growing solutions to \eqref{m22} of the form \eqref{lt4}.

\begin{theorem}
The functions $\{y_{i}(t,x)\}_{i=1}^{4}$ defined by
\begin{equation}\label{d20}
\begin{bmatrix}
y_{1}(t,x)\\[0.3em]
y_{2}(t,x)\\[0.3em]
y_{3}(t,x)\\[0.3em]
y_{4}(t,x)
\end{bmatrix}
=(Bt)^{1/4}\begin{bmatrix}
 0    & \frac{1}{\sqrt{2}}  & -1  &\frac{1}{\sqrt{2}} \\[0.5em]
 1    & -\frac{1}{\sqrt{2}} &  0  &\frac{1}{\sqrt{2}} \\[0.5em]
 0    & \frac{1}{\sqrt{2}}  &  1  &\frac{1}{\sqrt{2}} \\[0.5em]
 1    &\frac{1}{\sqrt{2}}   &  0  &-\frac{1}{\sqrt{2}}
\end{bmatrix}
\begin{bmatrix}
 \frac{1}{\Gamma\left(\frac{5}{4}\right)}z_{1}(u)\\[0.3em]
 \frac{1}{\Gamma(1)} z_{2}(u)\\[0.3em]
 \frac{1}{2\Gamma\left(\frac{3}{4}\right)}z_{3}(u)\\[0.3em]
 \frac{1}{6\Gamma\left(\frac{1}{2}\right)}z_{4}(u)
\end{bmatrix},
\end{equation}
\noindent where $u=\frac{x}{(Bt)^{1/4}}$, form a fundamental set of self-similar solutions to \eqref{m22}, namely
\begin{equation}
y_{t}+B y_{xxxx}=0,\quad x\in\mathbb{R}\setminus\{0\}, \quad t>0,\nonumber
\end{equation}
satisfying \eqref{lt4}, which are linearly independent over the domain of definition. Moreover for $t>0$,
\begin{eqnarray}
&&\label{d24a}
\lim_{x\rightarrow\infty}y_{1}(t,x)=\lim_{x\rightarrow\infty}y_{2}(t,x)=0,\\
&&\label{d24b}
\lim_{x\rightarrow-\infty}y_{3}(t,x)=\lim_{x\rightarrow-\infty}y_{4}(t,x)=0,\\
&&\label{d25a}
\lim_{x\rightarrow-\infty} y_{1}(t, x)=\lim_{x\rightarrow-\infty} y_{2}(t, x)=-\infty,\\
&&\label{d25b}
\lim_{x\rightarrow\infty} y_{3}(t,x)=\infty,\quad \lim_{x\rightarrow\infty}y_{4}(t,x)=-\infty,
\end{eqnarray}
and furthermore
\begin{eqnarray}
\label{d26a}&&\lim_{t\rightarrow0}y_{1}(t,x)=\lim_{t\rightarrow0}y_{2}(t,x)=0,\quad x>0,\\
\label{d26b}&&\lim_{t\rightarrow0}y_{3}(t,x)=\lim_{t\rightarrow0}y_{4}(t,x)=0, \quad x<0.
\end{eqnarray}
\end{theorem}
\begin{proof}
Let $y(t,x)$ be a self-similar solution of the form \eqref{lt4} to \eqref{m22}. Then formally taking the Laplace transform of \eqref{m22} with respect to the time variable $t$, we get

\begin{equation}
\label{m26}p\overline{y}+B\overline{y}_{xxxx}=0,\quad x\in\mathbb{R}\setminus\{0\},\quad p>0,
\end{equation}
\noindent where
\begin{equation}
\overline{y}(p,x)=\int_{0}^{\infty}e^{-pt}y(t,x)dt.\nonumber
\end{equation}

From the assumed self-similarity of $y(t,x)$, it follows from \eqref{d10} and \eqref{zij} that
\begin{equation}
\label{d44}\frac{\partial^{(i-1)}y}{\partial {x}^{(i-1)}}(t,0)=(Bt)^{(2-i)/4}C_{i}(i-1)!,\quad i=1,2,3,4.
\end{equation}
By taking the Laplace transform of \eqref{d44}, we get
\begin{equation}\label{lic1}
\frac{\partial^{(i-1)}\overline{y}}{\partial x^{(i-1)}}(p,0)=C_{i}(i-1)!B^{(2-i)/4}\ \Gamma\left(\frac{6-i}{4}\right)p^{(i-6)/4},\quad i=1,2,3,4.
\end{equation}
Let us now note that the ODE given in \eqref{m26} has a set of four fundamental solutions, $\{\overline{y}_{j}(p,x)\}_{j=1}^{4}$,

\begin{eqnarray}\label{d28}
&&\overline{y}_{1}(p,x)=B^{1/4}p^{-5/4}\exp\left(-\frac{p^{1/4}}{B^{1/4}\sqrt{2}}x\right)\sin\left(\frac{p^{1/4}}{B^{1/4}\sqrt{2}}x\right),\\
&&\label{d29}\overline{y}_{2}(p,x)=B^{1/4}p^{-5/4}\exp\left(-\frac{p^{1/4}}{B^{1/4}\sqrt{2}}x\right)\cos\left(\frac{p^{1/4}}{B^{1/4}\sqrt{2}}x\right),\\
&&\label{d30}\overline{y}_{3}(p,x)=B^{1/4}p^{-5/4}\exp\left(\frac{p^{1/4}}{B^{1/4}\sqrt{2}}x\right)\sin\left(\frac{p^{1/4}}{B^{1/4}\sqrt{2}}x\right),\\
&&\label{d31}\overline{y}_{4}(p,x)=B^{1/4}p^{-5/4}\exp\left(\frac{p^{1/4}}{B^{1/4}\sqrt{2}}x\right)\cos\left(\frac{p^{1/4}}{B^{1/4}\sqrt{2}}x\right).
\end{eqnarray}
It follows from \eqref{d28}-\eqref{d31} that

\begin{equation}
\frac{\partial^{(i-1)}\overline{y}_{j}(p,0)}{\partial x^{(i-1)}}=d^{\ast}_{ij}B^{(2-i)/4}p^{(i-6)/4},\quad i,j\in\{1,2,3,4\},\nonumber
\end{equation}
where $d^{\ast}_{ij}=[D^{T}]_{ij}=[D]_{ji}$, 

\begin{equation}
D:=\begin{bmatrix}
 0    & \frac{1}{\sqrt{2}}  & -1  &\frac{1}{\sqrt{2}} \\[0.5em]
 1    & -\frac{1}{\sqrt{2}} &  0  &\frac{1}{\sqrt{2}} \\[0.5em]
 0    & \frac{1}{\sqrt{2}}  &  1  &\frac{1}{\sqrt{2}} \\[0.5em]
 1    &\frac{1}{\sqrt{2}}   &  0  &-\frac{1}{\sqrt{2}}
\end{bmatrix}.\nonumber
\end{equation} 
These solutions may be linearly combined to yield

\begin{equation}\label{lt}
\overline{y}(p,x)=\sum_{i=1}^{4}c_{i}\overline{y}_{i}(p,x),
\end{equation}
which satisfies both \eqref{m26} and \eqref{lic1} if we set
\begin{equation}\label{d32}
\begin{bmatrix}
c_{1}\\[0.3em]
c_{2}\\[0.3em]
c_{3}\\[0.3em]
c_{4}
\end{bmatrix}
=\frac{1}{2}D
\begin{bmatrix}
0!\Gamma\left(\frac{5}{4}\right) C_{1}\\[0.3em]
1!\Gamma(1)C_{2}\\[0.3em]
2!\Gamma\left(\frac{3}{4}\right)C_{3}\\[0.3em]
3!\Gamma\left(\frac{1}{2}\right)C_{4}
\end{bmatrix}.
\end{equation}

Recalling that $\overline{y}(p,x)$ is Laplace transform of $y(t,x)$, we denote by $y_{i}(t,x)$ the inverse Laplace transform of $\overline{y}_{i}(p,x)$, for $i=1,2,3,4,$ and from \eqref{lt} we obtain that
\begin{equation}\label{nh0}
y(t,x)=\sum_{i=1}^{4}c_{i}y_{i}(t,x).
\end{equation}
Since $y(t,x)$ is a self-similar solution to \eqref{m22}, by \eqref{d10} it may be expressed equivalently as

\begin{equation}\label{d27}
y(t,x)=(Bt)^{1/4}\sum_{i=1}^{4}C_{i}z_{i}\left(\frac{x}{(Bt)^{1/4}}\right),\ x>0\ (\mbox{or }x<0),\ t>0,
\end{equation}
where the coefficients in \eqref{d27} can easily be obtained from \eqref{d32} upon noting that $\frac{1}{2}DD^{T}=I$.

 The proofs of the asymptotic properties \eqref{d24a}-\eqref{d25b} and \eqref{d26a}-\eqref{d26b} are rather technical and are given in Appendix \ref{de}.
\end{proof}

\section{Data Fitting}\label{sec:df}

A major advantage of our solution representations over previous solutions such as the solution given by Mullins \cite{Mullins1957} and the solution given in Amram et al.~\cite{Amram2014} is that it can be used effectively to do data fitting, enabling the identification of the effective boundary conditions and relevant physical parameters during thermal grooving. The solutions given  in \cite{Mullins1957} and \cite{Amram2014} were prescribed via power series with recursively defined coefficients. By relying on the power series representation, these solutions can be plotted as truncated series (or polynomials) with unbounded growth as $x\rightarrow\pm\infty$; accordingly, the resultant plots of these solutions are primarily helpful for analyzing to the surface profiles in close proximity to the thermal groove, as in Fig.~8a \cite{Amram2014}. The solutions \eqref{d41} and \eqref{nh0} which were derived here are more general and in particular, \eqref{d41} is prescribed in terms of known functions which can be readily and accurately evaluated using common software in an arbitrarily wide neighborhood of the thermal groove.

Below, we illustrate data fitting of our solution to experimental data by Amram et al.~\cite{Amram2014}, from atomic microscopy measurements of thermal groove formation in a nickel (Ni) film, Fig.~\ref{fig:film20}, and in bulk Ni, Fig.~\ref{fig:bulk20}, after annealing of the specimens at $700^{\circ}$C for 20 minutes.
\begin{figure}[h]
\centering
  \includegraphics[width=0.66\textwidth]{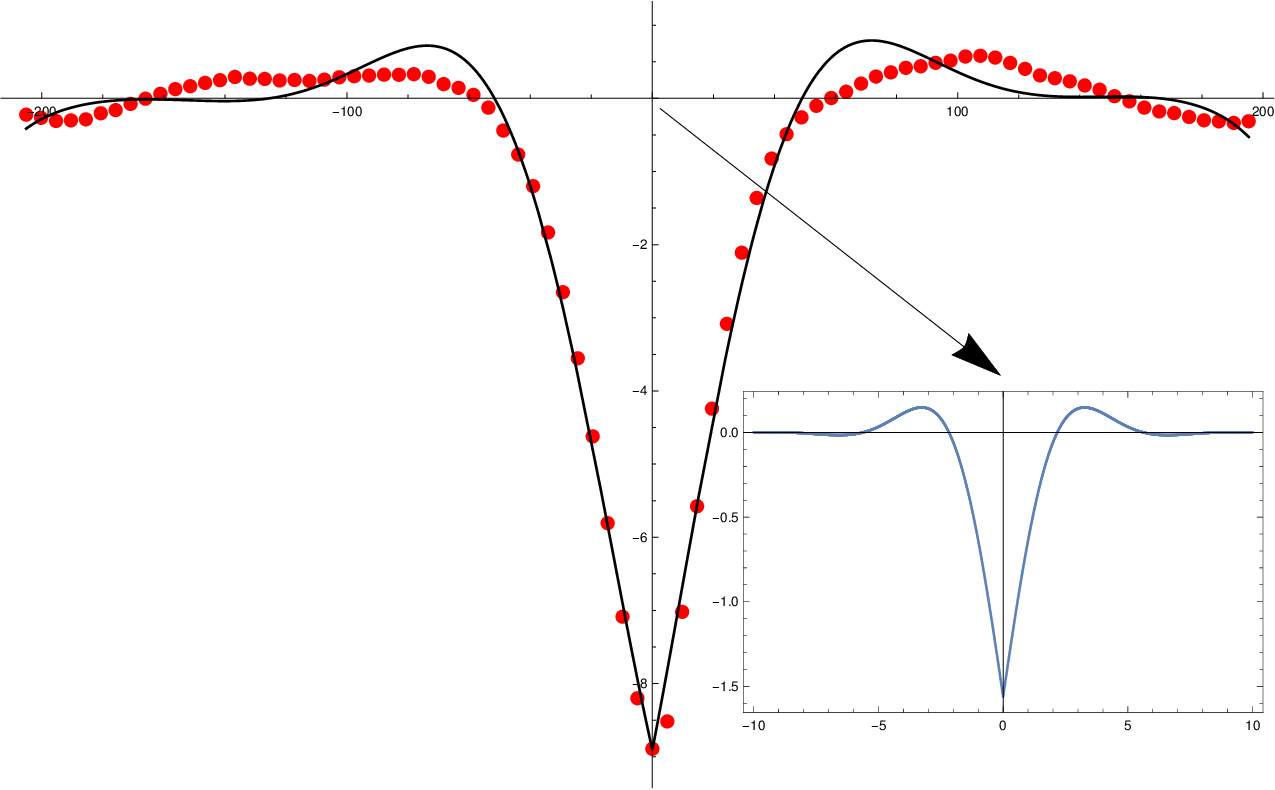}
\caption{\scriptsize{Experimental data from Amram et al., \cite{Amram2014}, of a groove profile in a thin Ni film annealed at $700^{\circ}$C for 20 min (red dots) plotted together with our solution (black line). The embedded figure portrays the calculated groove profile based on the solution (as a function of $u=\frac{x}{(Bt)^{1/4}}$) given in \cite{Amram2014}$^{\ref{acta}}$.}}
\label{fig:film20}
\end{figure}
\footnote{\label{acta}Reprinted from Acta Materialia, Vol. 69, D.~Amram, L.~Klinger, N.~Gazit, H.~Gluska, E.~Rabkin, Grain boundary grooving in thin films revisited: The role of interface diffusion, pp.~386-396, Copyright (2014), with permission from Elsevier.}

\begin{figure}[h]
\centering
  \includegraphics[width=0.9\textwidth]{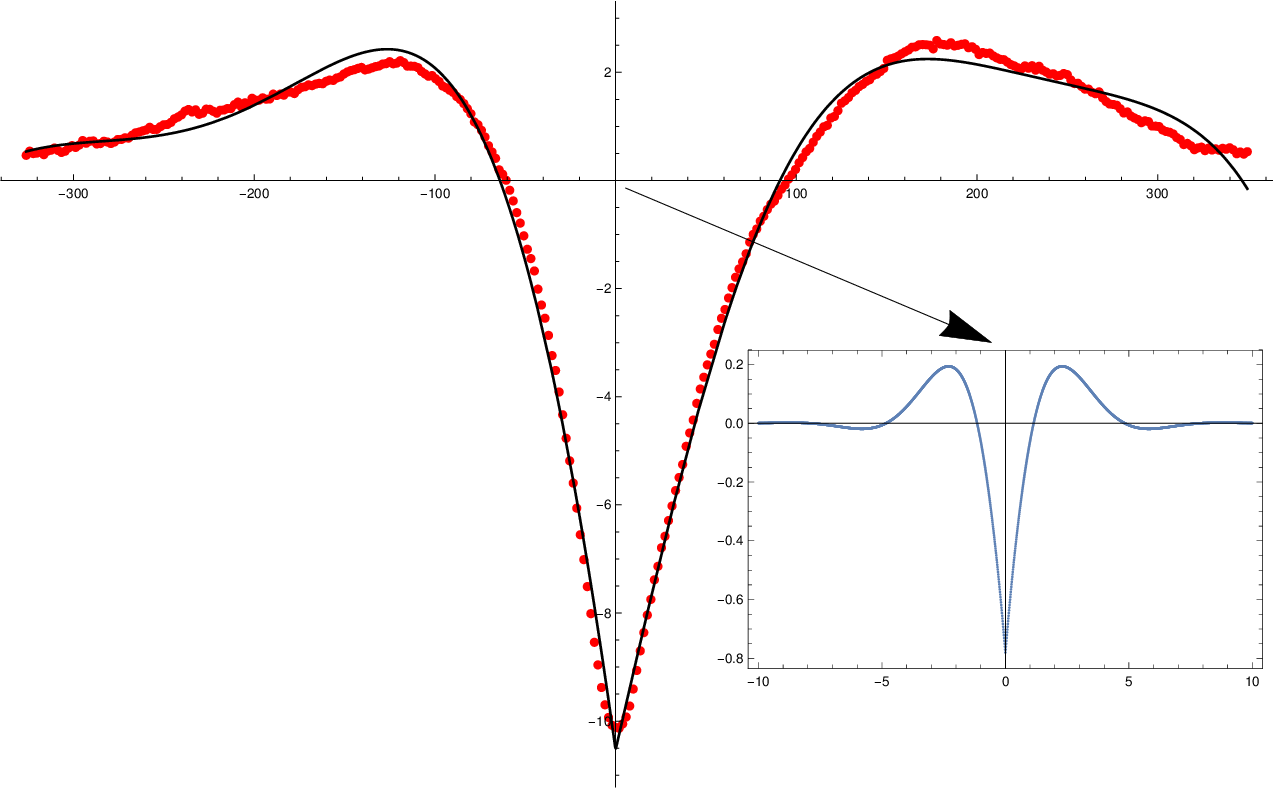}
\caption{\scriptsize{Experimental data from Amram et al., \cite{Amram2014}, of a groove profile in Ni bulk annealed at $700^{\circ}$C for 20 min (red dots) plotted together with our solution (black line). The embedded figure portrays the calculated groove profile based on Mullins' \cite{Mullins1957} solution (as a function of $u=\frac{x}{(Bt)^{1/4}}$)$^{\ref{acta}}$.}}
\label{fig:bulk20}
\end{figure}
\appendix
\setcounter{theorem}{0}
\section{Derivation of a Set of Fundamental Solutions}\label{dgs}

\begin{lemma} The equation \eqref{m34}
\begin{equation}
Z^{(4)}(u)-\frac{1}{4}uZ'(u)+\frac{1}{4}Z(u)=0\quad u>0\quad (\mbox{or }u<0)\nonumber
\end{equation}
can be transformed into a GHDE for $V=V(v)$,
\begin{equation}\label{a1}
v \frac{d}{dv}\Bigl( \prod_{i=1}^{q} (v\frac{d}{dv} + b_i -1) \Bigr)V -v \Bigl( \prod_{j=1}^{p} (v\frac{d}{dv} + a_j) \Bigr)V=0,
\end{equation}
where  $p,\,q \in \mathbb{Z}_+$, $p, \, q>2,$  and $\{a_{j}\}_{j=1}^{j=p}, \; \{b_{i}\}_{i=1}^{i=q}, v  \in \mathbb{C}$, by setting
\begin{equation}\label{a3}
p=1,\quad q=3,\quad \hbox{\, and \,}\quad a_{1}=-\frac{1}{4}, b_{1}=\frac{1}{4}, b_{2}=\frac{1}{2}, b_{3}=\frac{3}{4},
\end{equation}
and making the change of variable
\begin{equation}\label{a2}
v(u)=\frac{u^{4}}{256}.
\end{equation}
\end{lemma}
\begin{proof}
Observe that \eqref{a1} may be expanded and written as
\begin{equation}\label{a9}
v^{q}\frac{d^{q+1}V}{d v^{q+1}}+\sum_{j=1}^{q}v^{j-1}(\alpha_{j}v+\beta_{j})\frac{d^{j}V}{d v^{j}}+\alpha_{0}V=0\quad\mbox{ if }p\leq q,
\end{equation}
\noindent where $\alpha_{j}$, $\beta_{j}\in \mathbb{C}$, or as
\begin{equation}\label{a10}
v^{q}(1-v)\frac{d^{q+1}V}{d v^{q+1}}+\sum_{j=1}^{q}v^{j-1}(\tilde{\alpha}_{j}v+\tilde{\beta}_{j})\frac{d^{j}V}{d v^{j}}+\tilde{\alpha}_{0}V=0\quad\mbox{ if }p=q+1,
\end{equation}
\noindent where $\tilde{\alpha}_{j}$, $\tilde{\beta}_{j}\in\mathbb{C}$, or as
\begin{equation}\label{a10a}
v^{p}\frac{d^{p}V}{d v^{p}}+\sum_{j=1}^{p-1}v^{j-1}(\tilde{\tilde{\alpha}}_{j}v+\tilde{\tilde{\beta}}_{j})\frac{d^{j}V}{d v^{j}}+\tilde{\tilde{\alpha}}_{0}V=0\quad\mbox{ if }p> q+1,
\end{equation}
where $\tilde{\tilde{\alpha}}_{j}$, $\tilde{\tilde{\beta}}_{j}$ $\in \mathbb{C}$, \cite{NIST}. Note that if we set $q=3$ in (\ref{a9}), or (\ref{a10}), or we set $p=4$ in \eqref{a10a}, then the resultant equations are fourth order  linear homogeneous ODEs in which the coefficients of $\frac{d^{j}}{d v^{j}}V$ for $j=0,\ldots, 4,$ are polynomials in $v$ of degree $d_j$ with  $d_j \le j$. This allows us to postulate that via a suitable change of variables, $u=u(v)$, equation \eqref{m34} can be transformed into either \eqref{a9}, \eqref{a10} or \eqref{a10a} for some suitable choice of the parameters $p$, $q\in\mathbb{Z}_{+}$ and $\alpha_{j}$, $\beta_{j}$, $\tilde{\alpha}_{j}$, $\tilde{\beta}_{j}$, $\tilde{\tilde{\alpha}}_{j}$, $\tilde{\tilde{\beta}}_{j}\in\mathbb{R}$, since \eqref{m34} is an ordinary differential equation for $Z(u), u\in\mathbb{R}$ with real valued coefficients.

 Let us first consider the case   $q=3$ and $p=q+1=4$, which corresponds to (\ref{a10}), and let us attempt to find a change of variables which can transform \eqref{m34} into \eqref{a10}. Setting $u= u(v)$ in \eqref{m34} yields for $V(v)=Z(u)|_{u=u(v)}$ the equation
\begin{multline}\label{a11}
\frac{1}{(u')^{4}}\frac{d^{4}V}{dv^{4}}-6\frac{u''}{(u')^{5}}\frac{d^{3}V}{dv^{3}}+\left(15\frac{(u'')^{2}}{(u')^{6}}
-4\frac{u'''}{(u')^{5}}\right)\frac{d^{2}V}{dv^{2}}\\
+\left(10\frac{u''u'''}{(u')^{6}}-15\frac{(u'')^{3}}{(u')^{7}}-\frac{u^{(4)}}{(u')^{5}}-\frac{1}{4}\frac{u}{u'}\right)\frac{dV}{dv}+\frac{1}{4}V=0.
\end{multline}
\noindent Equating the coefficients of $\frac{d^{4}V}{dv^{4}}$ and $V$ in \eqref{a10} and \eqref{a11} implies that
\begin{equation}\label{a22}
(u')^{-4}=v^{3}-v^{4}.
\end{equation}
\noindent However by using \eqref{a22} to evaluate the coefficient of $\frac{d^{2}V}{dv^{2}}$ in \eqref{a11}, we find that
\begin{equation}
15\frac{(u'')^{2}}{(u')^{6}}-4\frac{u'''}{(u')^{5}}=-\frac{v(51-168v+112v^{2})}{16(v-1)},\nonumber
\end{equation}
\noindent which is not of the form $\tilde{\alpha}_{2}v^{2}+\tilde{\beta}_{2}v$ for $\tilde{\alpha}_{2}$, $\tilde{\beta}_{2}\in\mathbb{R}$. Hence we conclude that there does not exist a change of variables which transforms \eqref{m34} into \eqref{a10}.

 A similar argument allows us to conclude that also in the case $p > q + 1$, which corresponds to \eqref{a10a}, there is no change of variables which can transform \eqref{m34} into \eqref{a10a}. This can be seen as follows. In accordance with the form of \eqref{a10a} we equate the coefficients of $\frac{d^{4}V}{dv^{4}}$ and $V$ in \eqref{a10a} and \eqref{a11}, which implies
\begin{equation}\label{a10b}
(u')^{-4}=v^{4}.
\end{equation}
\noindent Using \eqref{a10b} to calculate the coefficient of $\frac{dV}{dv}$ in \eqref{a11}, we get
\begin{equation}
10\frac{u''u'''}{(u')^{6}}-15\frac{(u'')^{3}}{(u')^{7}}-\frac{u^{(4)}}{(u')^{5}}-\frac{1}{4}\frac{u}{u'}=v-\frac{v}{4}(\ln(v)+C), \quad C\in\mathbb{R},\nonumber
\end{equation}
\noindent which is not of the form $\tilde{\tilde{\alpha}}_{1}v+\tilde{\tilde{\beta}}_{1}$ for $\tilde{\tilde{\alpha}}_{1},\tilde{\tilde{\beta}}_{1}\in\mathbb{R}$. This implies that there is no change of variables that can transform \eqref{m34} into \eqref{a10a}.

 So let us now focus on the case $p\leq q =3$ which corresponds to (\ref{a9}), and let us look for a change of variable which can transform \eqref{m34} into \eqref{a9}. Equating the coefficients of $\frac{d^{4}}{dv^{4}}V$ in \eqref{a11} and in \eqref{a9}, we get that
$(u')^{-4}=v^{3},$
which implies that
\begin{equation} \label{cv0}
u(v)=4v^{1/4}+C,\quad C \in \mathbb{R}.
\end{equation}
Taking (\ref{cv0}) into account and matching the coefficients of $\frac{d^{j}}{dv^{j}}V,$ $j=1,2,3,$ in \eqref{a11} and \eqref{a9}, we get
\begin{eqnarray}
&&10\frac{u''u'''}{(u')^{6}}-15\frac{(u'')^{3}}{(u')^{7}}-\frac{u^{(4)}}{(u')^{5}}-\frac{1}{4}\frac{u}{u'}=\frac{3}{32}-\frac{C}{4}v^{3/4}-v=\alpha_{1}v+\beta_{1},\nonumber\\
&&15\frac{(u'')^{2}}{(u')^{6}}-4\frac{u'''}{(u')^{5}}=\frac{51}{16}v=\alpha_{2}v^{2}+\beta_{2}v,\quad -6\frac{u''}{(u')^{5}}=\frac{9}{2}v^{2}=\alpha_{3}v^{3}+\beta_{3}v^{2},\nonumber
\end{eqnarray}
which imply that
\begin{equation}
\alpha_{0}\!=\frac{1}{4},\quad \alpha_{1}\!=-1,\quad \alpha_{2}\!=\alpha_{3}\!=0,\quad \beta_{1}\!=\frac{3}{32},\quad \beta_{2}\!=\frac{51}{16},\quad\beta_{3}=\frac{9}{2},
\end{equation}
and that $C=0$ in (\ref{cv0}), which implies that
\begin{equation}\label{a12}
u(v)=4v^{1/4}.
\end{equation}

Next, we want to write \eqref{m34} in the form \eqref{a1}, since specific knowledge of the values of $a_{i}$ and $b_{i}$ in \eqref{a1} will provide us with a set of fundamental solutions to \eqref{m34}. Since \eqref{a1} is equivalent to \eqref{a9}, \cite[Section 16.8(ii)]{NIST}, in order to identify the coefficients in \eqref{a1}, we may proceed by using the inverse of the function $u=u(v)$ defined in \eqref{a12}, namely,
\begin{equation}\label{a13}
v(u)=\frac{1}{256}u^{4},
\end{equation}
\noindent as a change of variables in \eqref{a1}.

Recalling that $q=3$ and $p\leq q$, it follows that $p\in\{0, 1, 2, 3\}$. We demonstrate that $p=1$ by eliminating the other cases. The case $p=0$ is easily eliminated, since when $p=0$ equation \eqref{a1} yields
\begin{equation}\label{a13a}
\frac{d^{4}Z}{du^{4}}+B_{3}\frac{4}{u}\frac{d^{3}Z}{du^{3}}+B_{2}\frac{16}{u^{2}}\frac{d^{2}Z}{du^{2}}+B_{1}\frac{64}{u^{3}}\frac{dZ}{du}=0,
\end{equation}
which is not equivalent to \eqref{m34}, as the coefficients of $Z$ in \eqref{a13a} and \eqref{m34} do not match.

Next, let us suppose that $p=3$. Then \eqref{a1} yields

\begin{multline}\label{a14}
\frac{d^{4}Z}{du^{4}}+(B_{3}\frac{4}{u}+A_{3}\frac{u^{3}}{64})\frac{d^{3}Z}{du^{3}}+(B_{2}\frac{16}{u^{2}}+A_{2}\frac{u^{2}}{16})\frac{d^{2}Z}{du^{2}}\\
+(B_{1}\frac{64}{u^{3}}+A_{1}\frac{u}{4})\frac{dZ}{du}+A_{0}Z=0,
\end{multline}
\noindent where
\begin{eqnarray}
\label{a15}&&B_{1}=\frac{9}{16}(b_{1}+b_{2}+b_{3}-\frac{3}{4})-\frac{3}{4}(b_{1}b_{2}+b_{1}b_{3}+b_{2}b_{3})+b_{1}b_{2}b_{3},\\
\label{a16}&&B_{2}=b_{1}b_{2}+b_{1}b_{3}+b_{2}b_{3}-\frac{5}{4}(b_{1}+b_{2}+b_{3})+\frac{19}{16},\\
\label{a17}&&B_{3}=b_{1}+b_{2}+b_{3}-\frac{3}{2},\\
&&A_{0}=-a_{1}a_{2}a_{3},\nonumber\\
&&A_{1}=-\frac{1}{4}(a_{1}+a_{2}+a_{3}+\frac{1}{4})-(a_{1}a_{2}+a_{1}a_{3}+a_{2}a_{3}),\nonumber\\
&&A_{2}=-(a_{1}+a_{2}+a_{3}+\frac{3}{4}),\nonumber\\
\label{a18}&&A_{3}=-1.
\end{eqnarray}

\noindent However, comparing the coefficients of $\frac{d^{3}Z}{du^{3}}$ in equations \eqref{a14} and \eqref{m34} implies that
\begin{equation}
B_{3}\frac{4}{u}+A_{3}\frac{u^{3}}{64}=0,\nonumber
\end{equation}
\noindent which yields a contradiction, since $A_{3}=-1$, $B_{3}$ is a constant, and $u$ is a variable.

Suppose now that $p=2$. Then \eqref{a1} yields
\begin{equation}
\frac{d^{4}Z}{du^{4}}+B_{3}\frac{4}{u}\frac{d^{3}Z}{du^{3}}+(B_{2}\frac{16}{u^{2}}+\overline{A}_{2}\frac{u^{2}}{16})\frac{d^{2}Z}{du^{2}}+(B_{1}\frac{64}{u^{3}}+\overline{A}_{1}\frac{u}{4})\frac{dZ}{du}+\overline{A}_{0}Z=0,\nonumber
\end{equation}
\noindent where $B_{3}$, $B_{2}$, $B_{1}$ are as in \eqref{a15}-\eqref{a17}, and
\begin{equation}\label{a18a}
\overline{A}_{2}=-1,\quad \overline{A}_{1}=-(a_{1}+a_{2}+\frac{1}{4}),\quad \overline{A}_{0}=-a_{1}a_{2}.
\end{equation}
Matching the coefficient of $\frac{d^{2}Z}{du^{2}}$ gives
\begin{equation}
B_{2}\frac{16}{u^{2}}+\overline{A}_{2}\frac{u^{2}}{16}=0,\nonumber
\end{equation}
\noindent which again yields a contradiction.

Finally let us suppose that $p=1$. Then \eqref{a1} yields
\begin{equation}
\frac{d^{4}Z}{du^{4}}+B_{3}\frac{4}{u}\frac{d^{3}Z}{du^{3}}+B_{2}\frac{16}{u^{2}}\frac{d^{2}Z}{du^{2}}+\left(B_{1}\frac{64}{u^{3}}-\frac{u}{4}\right)\frac{dZ}{du}-{a}_{1}Z=0.\nonumber
\end{equation}
\noindent Matching the coefficients of $\frac{d^{i}Z}{du^{i}}$, $i=0, 1, 2, 3$ in the equation above and in \eqref{m34}, we get that $a_{1}=-\frac{1}{4}$ and that
\begin{eqnarray}
&&\label{a19}b_{1}+b_{2}+b_{3}-\frac{3}{2}=0,\\
&&\label{a20}b_{1}b_{2}+b_{1}b_{3}+b_{2}b_{3}-\frac{5}{4}(b_{1}+b_{2}+b_{3})+\frac{19}{16}=0,\\
&&\label{a21}\frac{9}{16}(b_{1}+b_{2}+b_{3}-\frac{3}{4})-\frac{3}{4}(b_{1}b_{2}+b_{1}b_{3}+b_{2}b_{3})+b_{1}b_{2}b_{3}=0.
\end{eqnarray}

\noindent Noting that equations \eqref{a19}-\eqref{a21} are invariant with respect to permutations of $\{b_{1}, b_{2}, b_{3}\}$, which reflects the fact that \eqref{a1} is similarly invariant, we find, modulo permutations, that

\begin{equation}
b_{1}=\frac{1}{4},\quad b_{2}=\frac{1}{2},\quad b_{3}=\frac{3}{4}.\nonumber
\end{equation}
\end{proof}

\section{Asymptotic Behavior of the Fundamental Solutions}\label{de}

In this appendix we give detailed proofs of the asymptotic properties \eqref{d24a}-\eqref{d25b} and \eqref{d26a}-\eqref{d26b} of the solutions $\{y_{i}(t,x)\}_{i=1}^{4}$; proof of the growth properties is given in Appendix \ref{asymp_grow} and proof of the decay properties is given in Appendix \ref{asymp_dec}. In Appendix \ref{asymp_dec}, the solution representations \eqref{msol}, \eqref{asol} are also derived.

\subsection{Asymptotically Growing Solutions}\label{asymp_grow}

\begin{lemma}\label{lem:h4} For $t> 0$,
\begin{equation}
\lim_{ x\rightarrow\infty} y_{3}(t,x)=\infty,\quad \lim_{ x\rightarrow\infty} y_{4}(t,x)=-\infty,\nonumber
\end{equation}
and for $x>0$,
\begin{equation}
\lim_{ t\rightarrow 0} y_{3}(t,x)=\infty,\quad \lim_{ t\rightarrow 0} y_{4}(t,x)=-\infty.\nonumber
\end{equation}

Similarly, for $t>0$,
\begin{eqnarray}
&&\lim_{ x\rightarrow-\infty} y_{1}(t,x)=\lim_{ x\rightarrow-\infty} y_{2}(t,x)=-\infty,\nonumber
\end{eqnarray}
and for $x>0$,
\begin{eqnarray}
&&\lim_{ t\rightarrow 0} y_{1}(t,-x)=\lim_{ t\rightarrow 0} y_{2}(t,-x)=-\infty.\nonumber
\end{eqnarray}
\end{lemma}
\begin{proof}[Proof of Lemma \ref{lem:h4}]
Since $z_{1}$, $z_{3}$ are even functions and $z_{2}$, $z_{4}$ are odd functions, it follows from the definitions of $y_{i}(t,x)$, $i=1,2,3,4,$ given in \eqref{d20}, that

\begin{eqnarray}
\label{j23}&&\quad\quad y_{3}(t,x)=(Bt)^{1/4}\left[\frac{z_{2}\big(\frac{x}{(Bt)^{1/4}}\big)}{\sqrt{2}}+\frac{z_{3}\big(\frac{x}{(Bt)^{1/4}}\big)}{2\Gamma\left(\frac{3}{4}\right)}+\frac{z_{4}\big(\frac{x}{(Bt)^{1/4}}\big)}{6\sqrt{2}\Gamma\left(\frac{1}{2}\right)}\right]\\
&&\quad\quad\quad=(Bt)^{1/4}\left[-\frac{z_{2}\big(-\frac{x}{(Bt)^{1/4}}\big)}{\sqrt{2}}+\frac{z_{3}\big(-\frac{x}{(Bt)^{1/4}}\big)}{2\Gamma\left(\frac{3}{4}\right)}-\frac{z_{4}\big(-\frac{x}{(Bt)^{1/4}}\big)}{6\sqrt{2}\Gamma\left(\frac{1}{2}\right)}\right]\nonumber\\
&&\quad\quad\quad=-y_{1}(t,-x),\quad t>0,\ x\in\mathbb{R}\setminus\{0\},\nonumber
\end{eqnarray}
and
\begin{multline}
\label{j24}y_{4}(t,x)=(Bt)^{1/4}\left[\frac{z_{1}\big(\frac{x}{(Bt)^{1/4}}\big)}{\Gamma\left(\frac{5}{4}\right)}+\frac{z_{2}\big(\frac{x}{(Bt)^{1/4}}\big)}{\sqrt{2}}-\frac{z_{4}\big(\frac{x}{(Bt)^{1/4}}\big)}{6\sqrt{2}\Gamma\left(\frac{1}{2}\right)}\right]\\
=(Bt)^{1/4}\left[\frac{z_{1}\big(-\frac{x}{(Bt)^{1/4}}\big)}{\Gamma\left(\frac{5}{4}\right)}-\frac{z_{2}\big(-\frac{x}{(Bt)^{1/4}}\big)}{\sqrt{2}}+\frac{z_{4}\big(-\frac{x}{(Bt)^{1/4}}\big)}{6\sqrt{2}\Gamma\left(\frac{1}{2}\right)}\right]\\
=y_{2}(t,-x),\quad t>0,\ x\in\mathbb{R}\setminus\{0\}.
\end{multline}

Hence it suffices to demonstrate the indicated asymptotic behavior for $y_{3}(t,x)$ and $y_{4}(t,x)$. 

From \eqref{j23}, \eqref{d37}-\eqref{d40}, and the expansion (\ref{def:hf}), it follows that
\begin{eqnarray}
y_{3}(t,x)&=&(Bt)^{1/4}\left[\frac{1}{\sqrt{2}} z_{2}(u)+\frac{1}{2\Gamma\left(\frac{3}{4}\right)}z_{3}(u)+\frac{1}{6\sqrt{2}\Gamma\left(\frac{1}{2}\right)}z_{4}(u)\right]\nonumber\\
&=&(Bt)^{1/4}\left[\frac{1}{\sqrt{2}} u+\frac{1}{2\Gamma\left(\frac{3}{4}\right)}\sum_{k=0}^{\infty}\frac{\left(\frac{1}{4}\right)_{k}}{\left(\frac{3}{4}\right)_{k}\left(\frac{5}{4}\right)_{k}\left(\frac{3}{2}\right)_{k}k!256^{k}}u^{4k+2}\right.\nonumber\\
&&\left.+\frac{1}{6\sqrt{2}\Gamma\left(\frac{1}{2}\right)}\sum_{k=0}^{\infty}\frac{\left(\frac{1}{2}\right)_{k}}{\left(\frac{5}{4}\right)_{k}\left(\frac{3}{2}\right)_{k}\left(\frac{7}{4}\right)_{k}k!256^{k}}u^{4k+3}\right],\nonumber
\end{eqnarray}
\noindent where $u=\frac{x}{(Bt)^{1/4}}$. Note that for $t, x>0$, $y_{3}(t,x)$ is a sum of positive terms of the form
\begin{equation}
c_{n}(Bt)^{1/4}u^{n},n\in\{1,4k+2,4k+3\ |\ k\in\mathbb{Z}_{+}\},\quad 0<c_{n}\in\mathbb{R}.\nonumber
\end{equation}
Thus $y_{3}(t,x)\rightarrow\infty$ as $x\rightarrow\infty$ for fixed $t>0$, and $y_{3}(t,x)\rightarrow\infty$ as $t\rightarrow 0$ for fixed $x>0$.

Next, note that

\begin{eqnarray}
y_{4}(t,x)&=&(Bt)^{1/4}\left[\frac{1}{\Gamma\left(\frac{5}{4}\right)}z_{1}(u)+\frac{1}{\sqrt{2}}z_{2}(u)-\frac{1}{6\sqrt{2}\Gamma\left(\frac{1}{2}\right)}z_{4}(u)\right]\nonumber\\
&=&(Bt)^{1/4}\left[\frac{1}{\Gamma\left(\frac{5}{4}\right)}+\frac{1}{\sqrt{2}}u-\frac{1}{6\sqrt{2}\Gamma\left(\frac{1}{2}\right)}u^{3}\right.\nonumber
\end{eqnarray}

\begin{eqnarray}
&&-\frac{1}{4\Gamma\left(\frac{5}{4}\right)}\sum_{k=1}^{\infty}\frac{\left(\frac{3}{4}\right)_{k-1}}{\left(\frac{1}{4}\right)_{k}\left(\frac{1}{2}\right)_{k}\left(\frac{3}{4}\right)_{k}k!256^{k}}u^{4k}\nonumber\\
&&\left.-\frac{1}{6\sqrt{2}\Gamma\left(\frac{1}{2}\right)}\sum_{k=1}^{\infty}\frac{\left(\frac{1}{2}\right)_{k}}{\left(\frac{5}{4}\right)_{k}\left(\frac{3}{2}\right)_{k}\left(\frac{7}{4}\right)_{k}k!256^{k}}u^{4k+3}\right].\nonumber
\end{eqnarray}
It follows from the expression above that for $t, x>0$, except for the first two terms, $y_{4}(t,x)$ can be expressed as a sum of negative terms of the form
\begin{equation}
c_{n}(Bt)^{1/4}u^{n},n\in\{3, 4k, 4k+3\ |\ k\in\mathbb{Z}_{+}\setminus\{0\}\},\quad 0>c_{n}\in\mathbb{R}.\nonumber
\end{equation}
The sum of the first three terms is negative for $u>4$. Hence $y_{4}(t,x)\rightarrow -\infty$ as $x\rightarrow\infty$ for fixed $t>0$, and $y_{4}(t,x)\rightarrow-\infty$ as $t\rightarrow 0$ for fixed $x>0$.
\end{proof}

\subsection{Asymptotically Decaying Solutions}\label{asymp_dec}
First we obtain integral representations for two linearly independent solutions of \eqref{m22}-\eqref{lic}, which we denote by $\tilde{y}_{1}(t,x)$ and $\tilde{y}_{2}(t,x)$, by using Fourier cosine transform and symmetry considerations. The solution $\tilde{y}_{1}(t,x)$ is proportional to Martin's integral representation for Mullins' series solution \cite{Martin2009,Mullins1957}. We  prove in detail that for any $t>0$,  both of these solutions tend to zero as $x \rightarrow \infty$, and that for any $x>0$, both of these solutions tend to zero as $t\rightarrow 0$. Then, by considering the boundary conditions satisfied by these solutions at $x=0$, we demonstrate that $y_1(t,x)$, $y_2(t,x)$ can both be expressed as linear combinations of $\tilde{y}_{1}(t,x)$, $\tilde{y}_{2}(t,x)$. This yields a closed form representation for the series solution obtained in Amram et al., and justifies \eqref{msol}, \eqref{asol}. Moreover, it  allows us to conclude that  $y_1(t,x)$ and $y_2(t,x)$  both tend to zero as $x \rightarrow \infty$ for $t>0$, and as $t\rightarrow 0$ for $x>0$.

In parallel with Lemma \ref{lem:h4}, we may summarize the asymptotic results obtained in this section as follows.

\begin{lemma}\label{lem:h3} Each solution $\{y_{i}(t,x)\}_{i=1}^{4}$ defined in \eqref{d20} exhibit the following  asymptotic decay,
\begin{equation}
\label{j26}\lim_{x\rightarrow\infty}y_{1}(t,x)=\lim_{x\rightarrow\infty}y_{2}(t,x)=0,\mbox{ for }t>0,
\end{equation}
and 
\begin{equation}
\label{j26n}\lim_{t\rightarrow 0}y_{1}(t,x)=\lim_{t\rightarrow 0}y_{2}(t,x)=0,\mbox{ for }x>0.
\end{equation}

Similarly
\begin{equation}
\label{j27}\lim_{x\rightarrow-\infty}y_{3}(t,x)=\lim_{x\rightarrow-\infty}y_{4}(t,x)=0,\mbox{ for }t>0,
\end{equation}
and 
\begin{equation}
\label{j27n}\lim_{t\rightarrow 0}y_{3}(t,x)=\lim_{t\rightarrow 0}y_{4}(t,x)=0,\mbox{ for }x<0.
\end{equation}
\end{lemma}

Recalling \eqref{j23}, \eqref{j24}, we note that it suffices to demonstrate the indicated asymptotic decay for $y_{1}(t,x)$ and $y_{2}(t,x)$.

Self-similar solutions of the form \eqref{lt4} to \eqref{m22}-\eqref{lic} can be found by utilizing the Fourier cosine transform under the assumption that
\begin{equation}\label{j2}
y_{xxx}(t,x), y_{xx}(t,x), y_{x}(t,x), y(t,x)\rightarrow0\mbox{ as }x\rightarrow\infty.
\end{equation}
The Fourier cosine transform was first used in this context by Martin \cite{Martin2009}, in conjunction with the condition
\begin{equation}\label{h0}
y'''(0,t)=0,\quad t>0.
\end{equation}
Here we proceed without imposing \eqref{h0}. The resultant solution can be expressed as a linear combination of two linearly independent solutions, denoted below as $\tilde{y}_{1}(t,x)$ and $\tilde{y}_{2}(t,x)$.

Taking the Fourier cosine transform of \eqref{m22}-\eqref{lic} with respect to the variable $x$, we get for solutions having the  similarity form \eqref{lt4} that
\begin{equation}\label{h1}
\frac{\partial Y_{c}}{\partial t}-B(Bt)^{-1/2}Z'''(0)+Bk^{2}Z'(0)+Bk^{4}Y_{c}=0,\ Y_{c}(0,k)=0,
\end{equation}
where
\begin{equation}
Y_{c}(t,k)=\int_{0}^{\infty}y(t,x)\cos(kx)dx=(Bt)^{1/4}\int_{0}^{\infty}Z\left(\frac{x}{(Bt)^{1/4}}\right)\cos(kx)dx.\nonumber
\end{equation}
\noindent Solving \eqref{h1} as an initial value problem in $t$, we obtain

\begin{equation}\label{h2}
Y_{c}(t,k)=Z'(0)\frac{(e^{-Btk^{4}}-1)}{k^{2}}+Z'''(0)\frac{e^{-Btk^{4}}}{k^{2}}\int_{0}^{Btk^{4}}e^{s}s^{-1/2}ds.
\end{equation}

By taking the inverse Fourier cosine transform of \eqref{h2}, we get
\begin{equation}\label{ycos}
y_{c}(t,x)=Z'(0)\tilde{y}_{1}(t,x)+Z'''(0)\tilde{y}_{2}(t,x)
\end{equation}
\noindent where
\begin{equation}
\label{h3a}\tilde{y}_{1}(t,x)=-\frac{2(Bt)^{1/4}}{\pi}\int_{0}^{\infty}\frac{(1-e^{-w^{4}})}{w^{2}}\cos\left(\frac{x}{(Bt)^{1/4}}w\right)dw,
\end{equation}  
\begin{multline}
\quad\quad\label{h3b}\tilde{y}_{2}(t,x)=\\
\frac{2(Bt)^{1/4}}{\pi}\int_{0}^{\infty}\frac{e^{-w^{4}}}{w^{2}}\left(\int_{0}^{w^{4}}e^{s}s^{-1/2}ds\right)\cos\left(\frac{x}{(Bt)^{1/4}}w\right)dw.
\end{multline}
We now prove that both of these solutions tend to zero as $x\rightarrow\infty$ for fixed $t>0$ as well as when $t\rightarrow 0$ for fixed $x>0$. For the proof of these properties for $\tilde{y}_{1}$ we make use of the following auxiliary proposition.

\begin{proposition}\label{prop:h1}The following hold
\begin{equation}\label{h13}
\lim_{w\rightarrow 0}\frac{1-e^{-w^{4}}}{w^{2}}=0,
\end{equation}
\begin{equation}\label{h14}
\lim_{w\rightarrow \infty}\frac{1-e^{-w^{4}}}{w^{2}}=0,
\end{equation}
\begin{equation}\label{h15}
\lim_{w\rightarrow 0}\left(\frac{1-e^{-w^{4}}}{w^{2}}\right)^{(k)}=\begin{cases}
    \frac{(-1)^{n}(4n+2)!}{(n+1)!}, &\text{if } k=4n+2,\ n\in\mathbb{Z}_{+},\\[1ex]
    0,  &\text{otherwise.}
  \end{cases}
\end{equation}
\begin{equation}\label{h16}
\lim_{w\rightarrow\infty}\frac{\left(\frac{1-e^{-w^{4}}}{w^{2}}\right)^{(k)}}{\frac{1}{w^{2}}}=0,\quad k\in\mathbb{Z}_{+}.
\end{equation}

\end{proposition}
\begin{proof}
The limits \eqref{h13}, \eqref{h14} follow from L'Hopital's rule.
We obtain \eqref{h15} by substituting the Maclaurin series for $e^{-w^{4}}$ into $\frac{1-e^{-w^{4}}}{w^{2}}$ and then calculating the derivatives at $w=0$.
The limit \eqref{h16} follows from the identity
\begin{equation}\label{h18}
\left(\frac{1-e^{-w^{4}}}{w^{2}}\right)^{(k)}=\frac{(-1)^{k}(k+1)!}{w^{k+2}}+e^{-w^{4}}\sum_{i=0}^{k}a_{i}(k)w^{-k-2+4i},
\end{equation}
where $a_{i}(k)$ are constants which depend only on $k$, which can be proved by mathematical induction on $k$. Substituting \eqref{h18} into the expression in \eqref{h16} and using the basic properties of the exponential function and power functions, we get
\begin{multline}
\lim_{w\rightarrow\infty}\frac{\left(\frac{1-e^{-w^{4}}}{w^{2}}\right)^{(k)}}{\frac{1}{w^{2}}}=\\
\lim_{w\rightarrow\infty}\left[
\frac{(-1)^{k}(k+1)!}{w^{k}}+e^{-w^{4}}\sum_{i=0}^{k}a_{i}(k)w^{-k+4(i-1)}\right]=0,\quad\nonumber
\end{multline}
for any $k\in\mathbb{Z}_{+}$.

\end{proof}

\begin{lemma}\label{lem:h1}
Let $\tilde{y}_{1}(t,x)$ be as defined in \eqref{h3a}. Then $\tilde{y}_{1}(t,x)$ tends to zero as $x\rightarrow\infty$ for fixed $t>0$ and as $t\rightarrow 0$ for fixed $x>0$.
\end{lemma}

\begin{proof}
We set $u=\frac{x}{(Bt)^{1/4}}$ as in Section \ref{sec:lt}, and integrate $\tilde{y}_{1}$ by parts $2k$ times

\begin{eqnarray}
&&-\frac{\pi}{2(Bt)^{1/4}}\tilde{y}_{1}(t,x)=\int_{0}^{\infty}\frac{1-e^{-w^{4}}}{w^{2}}\cos(uw)dw=\nonumber\\
&&=\frac{1}{u}\sin(uw)\left.\frac{1-e^{-w^{4}}}{w^{2}}\right\vert_{0}^{\infty}-\frac{1}{u}\int_{0}^{\infty}\sin(uw)\left(\frac{1-e^{-w^{4}}}{w^{2}}\right)'dw\nonumber\\
&&\quad(\mbox{by }\eqref{h13}\mbox{ and }\eqref{h14})\nonumber\\
&=&-\frac{1}{u}\int_{0}^{\infty}\sin(uw)\left(\frac{1-e^{-w^{4}}}{w^{2}}\right)'dw\nonumber
\end{eqnarray}

\begin{eqnarray}
&=&\ldots=\frac{(-1)^{k+1}}{u^{2k-1}}\sin(uw)\left.\left(\frac{1-e^{-w^{4}}}{w^{2}}\right)^{(2k-2)}\right\vert_{0}^{\infty}\nonumber\\
&&+(-1)^{k}\frac{1}{u^{2k-1}}\int_{0}^{\infty}\sin(uw)\left(\frac{1-e^{-w^{4}}}{w^{2}}\right)^{(2k-1)}dw\nonumber\\
&&\quad(\mbox{by }\eqref{h15}\mbox{ and }\eqref{h16})\nonumber
\end{eqnarray}
\begin{eqnarray}
&=&(-1)^{k}\frac{1}{u^{2k-1}}\int_{0}^{\infty}\sin(uw)\left(\frac{1-e^{-w^{4}}}{w^{2}}\right)^{(2k-1)}dw\nonumber\\
&=&\frac{(-1)^{k+1}}{u^{2k}}\cos(uw)\left.\left(\frac{1-e^{-w^{4}}}{w^{2}}\right)^{(2k-1)}\right\vert_{0}^{\infty}\nonumber\\
&&+(-1)^{k}\frac{1}{u^{2k}}\int_{0}^{\infty}\cos(uw)\left(\frac{1-e^{-w^{4}}}{w^{2}}\right)^{(2k)}dw\nonumber\\
&&(\mbox{by }\eqref{h15}\mbox{ and }\eqref{h16})\nonumber\\
\label{h4}&=&(-1)^{k}\frac{1}{u^{2k}}\int_{0}^{\infty}\cos(uw)\left(\frac{1-e^{-w^{4}}}{w^{2}}\right)^{(2k)}dw.
\end{eqnarray}
It follows  from \eqref{h15}, \eqref{h16} that the integrand in \eqref{h4} is bounded in $L^{1}((0,\infty))$ uniformly with respect to $u$, and therefore \eqref{h4} decays to zero faster than $u^{-n}$ for every positive integer $n$ as $u\rightarrow\infty$. Hence for fixed $t>0$, $\tilde{y}_{1}(t,x)\rightarrow0$  as $x\rightarrow\infty$ and for fixed $x>0$,  $\tilde{y}_{1}(t,x)\rightarrow0$ as $t\rightarrow 0$.
\end{proof}

\begin{lemma}\label{lem:h2}
Let $\tilde{y}_{2}(t,x)$ be as defined in \eqref{h3b}. Then $\tilde{y}_{2}(t,x)$ tends to zero as $x\rightarrow\infty$ for fixed $t>0$ as well as when $t\rightarrow 0$ for fixed $x>0$.
\end{lemma}

To prove Lemma \ref{lem:h2} we make use of the following proposition.

\begin{proposition}\label{prop:h2}Let $\erfi(z)$ denote the modified error function, \cite[Section 6.2.11]{Luke1969},
\begin{equation}
\erfi(z)=-i\erf(iz),\nonumber
\end{equation}
where
\begin{equation}
\erf(z)=\frac{2}{\sqrt{\pi}}\int_{0}^{z}e^{-t^{2}}dt,\quad z\in\mathbb{C}.\nonumber
\end{equation}

Then the following hold
\begin{equation}\label{h6}
\int_{0}^{w^{4}}e^{s}s^{-1/2}ds=\sqrt{\pi}\erfi(w^{2}),
\end{equation}

\begin{equation}\label{h7}
\lim_{w\rightarrow0}\left(\frac{e^{-w^{4}}\erfi(w^{2})}{w^{2}}\right)=\frac{2}{\sqrt{\pi}},
\end{equation}
\begin{equation}\label{h9}
\lim_{w\rightarrow 0}\left(\frac{e^{-w^{4}}\erfi(w^{2})}{w^{2}}\right)^{(k)}=\begin{cases}
    \frac{2^{n+1}(-1)^{n}(4n)!}{\sqrt{\pi}(2n+1)!},&\text{if } k=4n,\ n\in\mathbb{N},\\[1ex]
    0, &\text{otherwise,}
  \end{cases}
\end{equation}
\begin{equation}\label{h8}
\lim_{w\rightarrow\infty}\left(\frac{e^{-w^{4}}\erfi(w^{2})}{w^{2}}\right)=0,
\end{equation}
\begin{equation}\label{h11}
\lim_{w\rightarrow \infty}\frac{\left(\frac{e^{-w^{4}}\erfi(w^{2})}{w^{2}}\right)^{(k)}}{\frac{1}{w^{2}}}=0,\quad k\in\mathbb{Z}_{+}.
\end{equation}
\end{proposition}
\begin{proof}
The identity \eqref{h6} is obtained by applying the change of variable $-t^{2}=s$ in the definition of the $\erfi$ function.

The limits in \eqref{h7} and \eqref{h9} result from substituting the following Taylor series expansion
\begin{equation}\label{erfi:small}
\erfi(z)=\frac{2}{\sqrt{\pi}}\sum_{k=0}^{\infty}\frac{z^{2k+1}}{k!(2k+1)},
\end{equation}
which is valid for all $z\in\mathbb{C}$, see \cite[7.6.1]{NIST}, into $\frac{e^{-w^{4}}\erfi(w^{2})}{w^{2}}$.

The limits in \eqref{h8}, \eqref{h11} result from following technical claim.

\textbf{Claim:}
\begin{equation}\label{h19}
\lim_{w\rightarrow\infty}\frac{\left(\frac{e^{-w^{4}}\erfi(w^{2})}{w^{2}}\right)^{(n)}}{\frac{1}{w^{4+n}}}=\frac{(-1)^{n}(4)_{n}}{\sqrt{\pi}},\quad n=0,1,2,3,\ldots,
\end{equation}
where $(4)_{n}=\frac{(3+n)!}{3!}$ in accordance with the definition of the Pochhammer symbol.
\begin{proof}[Proof of the Claim:]
Both $\frac{e^{-w^{4}}\erfi(w^{2})}{w^{2}}$ and $\frac{1}{w^{4+n}}$ are smooth on $(0,\infty)$. First, we prove \eqref{h19} when $n=0$. Using L'Hopital's rule
\begin{equation}
\lim_{w\rightarrow\infty}\frac{\erfi(w^{2})}{\frac{e^{w^{4}}}{w^{2}}}=\lim_{w\rightarrow\infty}\frac{\frac{2}{\sqrt{\pi}}e^{w^{4}}2w}{\frac{4w^{3}e^{w^{4}}w^{2}-2we^{w^{4}}}{w^{4}}}=\frac{1}{\sqrt{\pi}}.\nonumber
\end{equation}
Hence
\begin{equation}\label{h20}
\lim_{w\rightarrow\infty}\frac{\frac{e^{-w^{4}}\erfi(w^{2})}{w^{2}}}{\frac{1}{w^{4}}}=\frac{1}{\sqrt{\pi}}.
\end{equation}
We apply L'Hopital's rule to \eqref{h20} and obtain
\begin{equation}
\frac{1}{\sqrt{\pi}}=\lim_{w\rightarrow\infty}\frac{\frac{e^{-w^{4}}\erfi(w^{2})}{w^{2}}}{\frac{1}{w^{4}}}=\lim_{w\rightarrow\infty}\frac{\left(\frac{e^{-w^{4}}\erfi(w^{2})}{w^{2}}\right)'}{\left(\frac{1}{w^{4}}\right)'}=\lim_{w\rightarrow\infty}\frac{\left(\frac{e^{-w^{4}}\erfi(w^{2})}{w^{2}}\right)'}{-\frac{4}{w^{5}}}\nonumber
\end{equation}
which gives us
\begin{equation}
\lim_{w\rightarrow\infty}\frac{\left(\frac{e^{-w^{4}}\erfi(w^{2})}{w^{2}}\right)'}{\frac{1}{w^{5}}}=-\frac{4}{\sqrt{\pi}}\nonumber
\end{equation}
namely \eqref{h19} when $n=1$.

We obtain \eqref{h19} by applying L'Hopital's rule to \eqref{h20} $n$ times.
\end{proof}
\end{proof}

\begin{proof}[Proof of Lemma \ref{lem:h2}]
Setting $u=\frac{x}{(Bt)^{1/4}}$ and integrating $\frac{\sqrt{\pi}}{2(Bt)^{1/4}}\tilde{y}_{2}(t,x)$ by parts $2k+1$ times, we get for $x>0$, $t>0$,

\begin{eqnarray}
&&\frac{\sqrt{\pi}}{2(Bt)^{1/4}}\tilde{y}_{2}(t,x)=\frac{1}{\sqrt{\pi}}\int_{0}^{\infty}\frac{e^{-w^{4}}}{w^{2}}\left(\int_{0}^{w^{4}}e^{s}s^{-1/2}ds\right)\cos(uw)dw,\nonumber\\
&&=\int_{0}^{\infty}\frac{e^{-w^{4}}}{w^{2}}\erfi(w^{2})\cos(uw)dw\quad(\mbox{by }\eqref{h6}),\nonumber
\end{eqnarray}

\begin{eqnarray}
&&=\left.\frac{\sin(uw)}{u}\frac{e^{-w^{4}}\erfi(w^{2})}{w^{2}}\right\vert_{0}^{\infty}-\int_{0}^{\infty}\frac{\sin(uw)}{u}\left(\frac{e^{-w^{4}}\erfi(w^{2})}{w^{2}}\right)'dw\nonumber\\
&&(\mbox{by }\eqref{h7}\mbox{ and }\eqref{h8}),\nonumber\\
&&=-\int_{0}^{\infty}\frac{\sin(uw)}{u}\left(\frac{e^{-w^{4}}\erfi(w^{2})}{w^{2}}\right)'dw,\nonumber\\
&&=\ldots=\frac{(-1)^{k+1}}{u^{2k}}\cos(uw)\left.\left(\frac{e^{-w^{4}}\erfi(w^{2})}{w^{2}}\right)^{(2k-1)}\right\vert_{0}^{\infty}\nonumber\\
&&+\frac{(-1)^{k}}{u^{2k}}\int_{0}^{\infty}\cos(uw)\left(\frac{e^{-w^{4}}\erfi(w^{2})}{w^{2}}\right)^{(2k)}dw\nonumber\\
&&(\mbox{by }\eqref{h9}\mbox{ and }\eqref{h11}),\nonumber\\
&&=\frac{(-1)^{k}}{u^{2k}}\int_{0}^{\infty}\cos(uw)\left(\frac{e^{-w^{4}}\erfi(w^{2})}{w^{2}}\right)^{(2k)}dw,\nonumber\\
&&=\frac{(-1)^{k+2}}{u^{2k+1}}\sin(uw)\left.\left(\frac{e^{-w^{4}}\erfi(w^{2})}{w^{2}}\right)^{(2k)}\right\vert_{0}^{\infty}\nonumber\\
&&+\frac{(-1)^{k+1}}{u^{2k+1}}\int_{0}^{\infty}\sin(uw)\left(\frac{e^{-w^{4}}\erfi(w^{2})}{w^{2}}\right)^{(2k+1)}dw\nonumber\\
&&(\mbox{by }\eqref{h9}\mbox{ and }\eqref{h11}),\nonumber
\end{eqnarray}

\begin{eqnarray}
\label{h10}&&=\frac{(-1)^{k+1}}{u^{2k+1}}\int_{0}^{\infty}\sin(uw)\left(\frac{e^{-w^{4}}\erfi(w^{2})}{w^{2}}\right)^{(2k+1)}dw.
\end{eqnarray}

Clearly, the function $\frac{e^{-w^{4}}\erfi(w^{2})}{w^{2}}$ is smooth on $(0, \infty)$, since it is the quotient of smooth functions and the denominator does not vanish. This together with \eqref{h9} and \eqref{h11}  imply  that the integrand in \eqref{h10} is bounded in $L^{1}((0,\infty))$ uniformly with respect to $u$. Hence \eqref{h10} tends to zero as $u\rightarrow\infty$.
\end{proof}

\begin{proof}[Proof of Lemma \ref{lem:h3}] 
As noted earlier, it suffices to prove \eqref{j26} and \eqref{j26n}, namely
\begin{equation}
\lim_{x\rightarrow\infty}y_{1}(t,x)=\lim_{x\rightarrow\infty}y_{2}(t,x)=0,\quad \mbox{ for fixed }\ t>0,\nonumber
\end{equation}
and that 
\begin{equation}
\lim_{t\rightarrow 0}y_{1}(t,x)=\lim_{t\rightarrow 0}y_{2}(t,x)=0,\quad \mbox{ for fixed }\ x>0.\nonumber
\end{equation}
In \cite{Martin2009} it is proved that $\tilde{y}_{1}(t,x)$,  which is defined in \eqref{h3a}, may be expressed as
\begin{equation}\label{h12}
(Bt)^{1/4}\sum_{n=0}^{\infty}a_{n}u^{n},\quad u=\frac{x}{(Bt)^{1/4}},
\end{equation}
where
\begin{equation}
a_{0}=-\frac{2}{\pi}\Gamma\left(\frac{3}{4}\right),\quad a_{1}=1,\quad a_{2}=-\frac{1}{4\pi}\Gamma\left(\frac{1}{4}\right),\quad a_{3}=0.\nonumber
\end{equation}
Since $\frac{\tilde{y}_{1}}{(Bt)^{1/4}}$ is a solution of the fourth order linear ordinary differential equation \eqref{m34}, the recursion relation
\begin{equation}\label{h103}
a_{n+4}=\frac{n-1}{4(n+1)(n+2)(n+3)(n+4)}a_{n},
\end{equation}
then determines the coefficients $a_{n}$ for $n\geq 4$. 

Note that $y_{1}$ and $y_{2}$, which are defined in \eqref{d20}, are also of the form \eqref{h12} and have coefficients that satisfy \eqref{h103}. Hence, we find
\begin{equation}\label{msol2}
\tilde{y}_{1}(t,x)=\frac{1}{\sqrt{2}}y_{1}(t,x)-\frac{1}{\sqrt{2}}y_{2}(t,x)
\end{equation}
by comparing the first four coefficients. Since $\tilde{y}_{1}(t,x)$ corresponds to Mullins' solution \cite{Mullins1957}, \eqref{msol2} implies \eqref{msol}. Similarly \eqref{asol} can be demonstrated directly by comparing the first four coefficients in the solution given in \cite{Amram2014} and in $-\frac{m}{\sqrt{2}}y_{2}(t,x)$, as implied by \eqref{d20}.

Next, we want to express $\tilde{y}_{2}(t,x)$ as
\begin{equation}\label{j19}
\tilde{y}_{2}(t,x)=\sum_{i=1}^{4}\alpha_{i}y_{i}(t,x),
\end{equation}
for $\alpha_{i}\in\mathbb{R}$, $i=1,2,3,4$.  Let us recall that we obtained $\tilde{y}_{2}(t,x)$ by solving \eqref{m22}-\eqref{lic}, under the assumption that $\tilde{y}_{2}(t,x)$ is of the form \eqref{lt4} with
\begin{equation}
\label{ic1}Z'(0)=0,\quad Z'''(0)=1.
\end{equation}
The conditions in \eqref{ic1} imply the following equalities
\begin{equation}\label{j5}
\frac{\alpha_{1}-\alpha_{2}+\alpha_{3}+\alpha_{4}}{\sqrt{2}}=Z'(0)=0,
\end{equation}
\begin{equation}\label{j30}
\frac{\alpha_{1}}{\sqrt{2}\Gamma\left(\frac{1}{2}\right)}+\frac{\alpha_{2}}{\sqrt{2}\Gamma\left(\frac{1}{2}\right)}+\frac{\alpha_{3}}{\sqrt{2}\Gamma\left(\frac{1}{2}\right)}-\frac{\alpha_{4}}{\sqrt{2}\Gamma\left(\frac{1}{2}\right)}=Z'''(0)=1.
\end{equation}

We can calculate $\tilde{y}_{2}(t,0)$ and $\frac{\partial^{2}\tilde{y}_{2}}{\partial x^{2}}(t,0)$ directly from \eqref{h3b}. Using \eqref{h6}, \eqref{erfi:small}, as well as the series expansions for $e^{-w^{4}}$ and $\prescript{}{2}{F}_{2}^{}$, we obtain that
\begin{eqnarray}
\tilde{y}_{2}(t,0)&=&\frac{2(Bt)^{1/4}}{\sqrt{\pi}}\int_{0}^{\infty}e^{-w^{4}}w^{-2}\erfi(w^{2})dw\nonumber\\
&=&\frac{4(Bt)^{1/4}w}{\pi}\prescript{}{2}{F}_{2}^{}(\frac{1}{4},1;\frac{5}{4},\frac{3}{2};-w^{4})\bigg|_{0}^{\infty}\nonumber\\
\label{nh1}&=&\frac{2}{\sqrt{\pi}}(Bt)^{1/4}\Gamma\left(\frac{3}{4}\right),
\end{eqnarray}
and
\begin{eqnarray}
\frac{\partial^{2}\tilde{y}_{2}}{\partial x^{2}}(t,0)&=&-\frac{2(Bt)^{-1/4}}{\sqrt{\pi}}\int_{0}^{\infty}e^{-w^{4}}\erfi(w^{2})dw\nonumber\\
&=&-\frac{4(Bt)^{-1/4}w^{3}}{3\pi}\prescript{}{2}{F}_{2}^{}(\frac{3}{4},1;\frac{3}{2},\frac{7}{4};-w^{4})\bigg|_{0}^{\infty}\nonumber\\
\label{nh3}&=&-\frac{2}{\sqrt{\pi}}(Bt)^{-1/4}\Gamma\left(\frac{5}{4}\right),
\end{eqnarray}
where the asymptotic evaluations of $\prescript{}{2}{F}_{2}^{}$ can be found in \cite[16.11(ii)]{NIST}. Both \eqref{nh1} and \eqref{nh3} can be verified by Mathematica. Combining \eqref{nh1} and \eqref{nh3} with \eqref{j19}, we get
\begin{equation}\label{j6}
\frac{\alpha_{2}(Bt)^{1/4}}{\Gamma\left(\frac{5}{4}\right)}+\frac{\alpha_{4}(Bt)^{1/4}}{\Gamma\left(\frac{5}{4}\right)}=\tilde{y}_{2}(t,0)=\frac{2}{\sqrt{\pi}}(Bt)^{1/4}\Gamma\left(\frac{3}{4}\right),
\end{equation}
\begin{equation}\label{j31}
-\frac{\alpha_{1}(Bt)^{-1/4}}{\Gamma\left(\frac{3}{4}\right)}+\frac{\alpha_{3}(Bt)^{-1/4}}{\Gamma\left(\frac{3}{4}\right)}=\frac{\partial^{2}\tilde{y}_{2}}{\partial x^{2}}(t,0)=-\frac{2}{\sqrt{\pi}}(Bt)^{-1/4}\Gamma\left(\frac{5}{4}\right).
\end{equation}

Solving \eqref{j5}, \eqref{j30}, \eqref{j6} and \eqref{j31}, we get
\begin{equation}
\alpha_{1}=\alpha_{2}=\frac{\sqrt{\pi}}{\sqrt{2}},\quad \alpha_{3}=\alpha_{4}=0.\nonumber
\end{equation}
Thus for $x>0$
\begin{equation}\label{fcsol2}
\tilde{y}_{2}(t,x)=\sqrt{\pi}\left(\frac{1}{\sqrt{2}}y_{1}(t,x)+\frac{1}{\sqrt{2}}y_{2}(t,x)\right).
\end{equation}

From \eqref{msol2}, \eqref{fcsol2}, it follows that
\begin{eqnarray}
&&y_{1}(t,x)=\frac{1}{\sqrt{2}}\tilde{y}_{1}(t,x)+\frac{1}{\sqrt{2\pi}}\tilde{y}_{2}(t,x),\nonumber\\
&&y_{2}(t,x)=-\frac{1}{\sqrt{2}}\tilde{y}_{1}(t,x)+\frac{1}{\sqrt{2\pi}}\tilde{y}_{2}(t,x),\nonumber
\end{eqnarray}
and the initial and far field properties of $y_{1}(t,x)$, $y_{2}(t,x)$ are implied by the results in Lemma \ref{lem:h1} and \ref{lem:h2}.
\end{proof}

\section*{Acknowledgments} The authors would like to thank Prof.~Eugen Rabkin for fruitful discussions regarding the experimental data, \cite{Amram2014}, that he generously shared. The authors would also like to thank Dr.~Orestis Vantzos for developing the code for the data fitting. 

The authors would like to acknowledge support from the Israel Science Foundation (Grant $\#1200/16$).

\bibliography{bibl_hvk14}
\bibliographystyle{amsplain}

\end{document}